\documentclass[11pt]{article}
\usepackage{latexsym,amsmath,amssymb}

\newif\ifdviwin

\usepackage[center]{titlesec}
\usepackage[english]{babel}
\usepackage{indentfirst}
\usepackage[mathscr]{eucal}
\usepackage{amssymb,amsmath,amsfonts}
\usepackage{fancybox,fancyhdr}
\usepackage{graphicx}
\usepackage[utf8]{inputenc}
\usepackage{float}
\usepackage{color}
\usepackage[pdftex]{hyperref}
\hypersetup{colorlinks=true,linkcolor=blue,citecolor=blue} 
\usepackage[export]{adjustbox}

\newif\ifdviwin

\dviwintrue

\usepackage[center]{titlesec}
\usepackage[english]{babel}
\usepackage{indentfirst}
\usepackage[mathscr]{eucal}
\usepackage{amssymb,amsmath,amsfonts}
\usepackage{fancybox,fancyhdr}
\usepackage{graphicx}
\usepackage[utf8]{inputenc}
\usepackage{float}
\usepackage{color}
\usepackage[pdftex]{hyperref}
\hypersetup{colorlinks=true,linkcolor=blue,citecolor=blue} 
\usepackage[export]{adjustbox}

\newif\ifdviwin

\dviwintrue

\let\infty=\infty \let\0=\emptyset  
\let\tilde=\widetilde

\let\varepsilon=\varepsilon

\def\R{\mathbb{R}}
\def\M{\mathbb{M}}
\def\H{\mathcal{H}}
\def\hh{\mathfrak{h}}
\def\hl{\mathfrak{h}_\lambda}

\def\D{\mathbb{D}}

\def\S{\mathbb{S}}
\def\m2r{\mathbb{M}^2\times\R}
\def\h2r{\mathbb{H}^2\times\R}
\def\mkr{\M^2(\kappa)\times\R}
\def\s2r{\mathbb{S}^2\times\R}

\def\Hss{\mathcal{H}\text{-}\mathrm{surfaces}}

\def\sig{\Sigma}
\def\r3{\mathbb{R}^3}

\def\te{\Theta_\varepsilon}
\def\t1{\Theta_1}
\def\tm1{\Theta_{-1}}

\newtheorem{defi}{Definition}[section]
\newtheorem{teo}[defi]{Theorem}
\newtheorem{prop}[defi]{Proposition}
\newtheorem{cor}[defi]{Corollary}
\newtheorem{lem}[defi]{Lemma}

\newenvironment{proof}{\rm \trivlist \item[\hskip \labelsep{\it
		Proof}:]}{\nopagebreak \hfill $\Box$ \endtrivlist}

\newenvironment{prooft1}{\rm \trivlist \item[\hskip \labelsep{\it
		Proof of Theorem \ref{teoremasaleeje}}:]}{\nopagebreak \hfill $\Box$ \endtrivlist}
		
\newenvironment{prooft2}{\rm \trivlist \item[\hskip \labelsep{\it
		Proof of Theorem \ref{teoremafueraeje}}:]}{\nopagebreak \hfill $\Box$ \endtrivlist}

\numberwithin{equation}{section}

\begin{document}
	
	%\mbox{}\vspace{0.4cm}
	
\begin{center}	
%\rule{16cm}{1.5pt}\vspace{0.5cm}
	\renewcommand{\thefootnote}{\,}
	{\large \bf Surfaces with prescribed mean curvature in $\mathbb{H}^2\times\mathbb{R}$
	\footnote{\hspace{-.75cm}
	}}\\
	\vspace{0.5cm} { Antonio Bueno$^\dagger$, Irene Ortiz$^\ddagger$}\\
	%\rule{16cm}{1.5pt}
\end{center}
\vspace{.5cm}
$^\dagger$Departamento de Ciencias e Informática, Centro Universitario de la Defensa de San Javier, E-30729 Santiago de la Ribera, Spain. \\ \vspace{.3cm}
\emph{E-mail address:} antonio.bueno@cud.upct.es\\
$^\ddagger$Departamento de Ciencias e Informática, Centro Universitario de la Defensa de San Javier, E-30729 Santiago de la Ribera, Spain. \\ \vspace{.3cm}
\emph{E-mail address:} irene.ortiz@cud.upct.es
	
\begin{abstract}
In this paper we study rotational surfaces in the space $\h2r$ whose mean curvature is given as a prescribed function of their angle function. These surfaces generalize, among others, the ones of constant mean curvature and the translating solitons of the mean curvature flow. Using a phase plane analysis we construct entire rotational graphs, catenoid-type surfaces, and exhibit a classification result when the prescribed function is linear.
\end{abstract}

\section{\large Introduction}
\vspace{-.5cm}

Let us consider a function $\H\in C^1(\S^2)$. An oriented surface $\sig$ immersed into $\r3$ is said to be a surface with \emph{prescribed mean curvature} $\H$ if its mean curvature function $H_\sig$ satisfies
\begin{equation}\label{PMC}
H_\sig(p)=\H(N_p)
\end{equation}
for every point $p\in\sig$, where $N$ denotes the \emph{Gauss map} of $\sig$. For short, we say that $\sig$ is an $\H$\emph{-surface}.  Let us observe that when the function $\H$ is chosen as a constant $H_0$, the surfaces defined by Equation \eqref{PMC} are just the surfaces with constant mean curvature (CMC) equal to $H_0$. 

The definition of this class of immersed surfaces has its origins on the famous Christoffel and Minkowski problems for ovaloids \cite{Chr,Min}. The existence of prescribed mean curvature ovaloids was studied, among others, by Alexandrov, Pogorelov and Guan-Guan \cite{Ale,Pog,GuGu}, while the uniqueness in the Hopf sense has been recently achieved by Gálvez and Mira \cite{GaMi1}. In this fashion, the first author jointly with Gálvez and Mira started to develop the \emph{global theory of surfaces with prescribed mean curvature in $\R^3$}, taking as motivation the fruitful theory of CMC surfaces, see \cite{BGM1,BGM2}. The first author also proved the existence and uniqueness to the Björling problem \cite{Bue3} and obtained half space theorems for $\Hss$ \cite{Bue4}.

A relevant case of $\Hss$ appears when the function $\H$ depends only on the height of the sphere, i.e., if there exists a one dimensional function $\hh\in C^1([-1,1])$ such that $\H(x)=\hh(\langle x,e_3\rangle)$ for every $x\in \S^2$. In this case, $\H$ is called \emph{rotationally symmetric} and Equation \eqref{PMC} reads as
\begin{equation}\label{PMCrotsim}
H_\sig(p)=\hh(\langle N_p,e_3\rangle)=\hh(\nu(p)),
\end{equation}
for every point $p\in\sig$, where $\nu(p):=\langle N_p,e_3\rangle$ is the so-called \emph{angle function} of $\sig$. In this setting, the authors have obtained a classification result for rotational $\H$-hypersurfaces in $\R^{n+1}$ satisfying \eqref{PMCrotsim}, whose prescribed function $\hh$ is \emph{linear}, i.e. $\hh(y)=ay+\lambda,\ a,\lambda\in\R$, see \cite{BuOr}.

Note that for defining Equation \eqref{PMCrotsim} we only need to measure the projection of a unit normal vector field onto a Killing vector field. Consequently, surfaces obeying \eqref{PMCrotsim} can be also defined in the product spaces $\m2r$, where $\M^2$ is a complete surface, as follows:

%Note that for defining Equation \eqref{PMCrotsim} we only need the measurement of a unit vector field onto a Killing vector field. As a very particular case, surfaces obeying Equation \eqref{PMCrotsim} have been already defined in the product spaces $\m2r$ as follows:

\begin{defi}\label{defisup}
Let $\hh$ be a $C^1$ function on $[-1,1]$. An oriented surface $\sig$ immersed in $\m2r$ has \emph{prescribed mean curvature} $\hh$ if its mean curvature function $H_\sig$ satisfies
	\begin{equation}\label{PMCH2R}
		H_\sig(p)=\hh(\langle\eta_p,\partial_z\rangle),
	\end{equation}
for every point $p\in\sig$, where $\eta$ is a unit normal vector field on $\sig$ and $\partial_z$ is the unit vertical Killing vector field on $\m2r$. 
\end{defi}
Again, note that $\nu(p):=\langle\eta_p,\partial_z\rangle$ is the angle function of $\sig$.
In analogy with the Euclidean case, we will simply say that $\sig$ is an $\hh$\emph{-surface}.

Observe that if $\hh\equiv H_0\in\R$, we recover the theory of CMC surfaces in $\m2r$, which experimented an extraordinary development since Abresch and Rosenberg \cite{AbRo} defined a \emph{holomorphic quadratic differential} on them, that vanishes on rotational examples. Also, if $\hh(y)=y$, the $\hh$-surfaces arising are the translating solitons of the mean curvature flow, see \cite{Bue1,Bue2,LiMa}. For a general function $\hh\in C^1([-1,1])$ under necessary and sufficient hypothesis, the first author obtained a Delaunay-type classification result in $\m2r$ \cite{Bue5}, and a structure-type result for properly embedded $\hh$-surfaces in $\h2r$ \cite{Bue6}.

Inspired by the fruitful theory of $\Hss$ in $\r3$, the purpose of this paper is to further investigate the theory of surfaces satisfying \eqref{PMCH2R} in the product space $\h2r$. The $\hh$-surfaces studied in this paper are motivated by the examples arising in the theory of minimal surfaces and translating solitons in both $\r3$ and $\h2r$, and by $\H$-hypersurfaces in $\R^{n+1}$ where $\H\in C^1(\S^n)$ is a linear function.

The rest of the introduction is devoted to detail the organization of the paper and highlight some of the main results.

In Section \ref{sec2} we deduce the formulae that the profile curve of a rotational $\hh$-surface in $\h2r$ satisfy. The resulting ODE will be treated as a non-linear autonomous system, and its qualitative study will be carried out by developing a phase plane analysis. From the previous work \cite{BGM2} we compile the main features of the phase plane adapted to the space $\h2r$. In Corollary \ref{ejefase} we prove the existence of two unique rotational $\hh$-surfaces, $\sig_+$ and $\sig_-$, intersecting orthogonally the axis of rotation with unit normal $\pm\partial_z$, respectively.

In Section \ref{sec3} we prove in Proposition \ref{bowls} the existence of entire, strictly convex $\hh$-graphs, called $\hh$-bowls in analogy with translating solitons. In Proposition \ref{hcats} we construct properly embedded $\hh$-annuli, called $\hh$-catenoids for their resemblance to the usual minimal catenoids in $\r3$ and $\h2r$.

In Section \ref{sec4} we focus on $\hh_\lambda$-surfaces, which are defined to be those whose prescribed function is linear, i.e. $\hh_\lambda(y):=ay+\lambda,\ a,\lambda\in\R$. The relevance of $\hl$-surfaces is that they satisfy certain characterizations that are closely related with the theory of manifolds with density. For instance, $\hl$-surfaces have \emph{constant weighted mean curvature} equal to $\lambda$ with respect to the density $e^\phi$, where $\phi(x)=a\langle x,\partial_z\rangle,\ \forall x\in\h2r$. In particular, if $\lambda=0$ we recover the fact that translating solitons are weighted minimal as pointed out by Ilmanen \cite{Ilm}. Also, $\hl$-surfaces are critical points for the weighted area functional, under compactly supported variations preserving the weighted volume (see \cite{BCMR} and \cite{BuOr}).

In this section we obtain our two main results in which we achieve a classification of complete rotational $\hl$-surfaces in $\h2r$. First, we prove that we can reduce to study the case $\hl(y)=y+\lambda,\ \lambda>0$. If such $\hl$-surfaces intersect the axis of rotation we get: 

\begin{teo}\label{teoremasaleeje}
Let be $\Sigma_+$ and $\Sigma_-$ the complete, rotational $\hl$-surfaces in $\h2r$ intersecting the rotation axis with upwards and downwards orientation respectively. Then: 
	
	\begin{itemize}
		\item[1.] For $\lambda>1/2$, $\Sigma_+$ is properly embedded, simply connected and converges to the flat CMC cylinder $C_\lambda$ of radius $\arg\tanh\left(\frac{1}{2\lambda}\right)$. Moreover:
		\begin{itemize}
			\item[1.1.] If $\lambda>\sqrt{2}/2$, $\sig_+$ intersects $C_\lambda$ infinitely many times.
			
			\item[1.2.] If $\lambda=\sqrt{2}/2$, $\sig_+$ intersects $C_\lambda$ a finite number of times and is a graph outside a compact set.
			
			\item[1.3.] If $\lambda<\sqrt{2}/2$, $\sig_+$ is a strictly convex graph over the disk in $\mathbb{H}^2$ of radius $\arg\tanh\left(\frac{1}{2\lambda}\right)$.
		\end{itemize}
		
		\item[2.] For $\lambda\leq 1/2$, $\sig_+$ is an entire, strictly convex graph.
		\item[3.] For $\lambda>\sqrt{5}/2$, $\sig_-$ is properly immersed (with infinitely many self-intersections), simply connected and has unbounded distance to the rotation axis.
		\item[4.] For $\lambda\leq\sqrt{5}/2$, $\sig_-$ is an entire graph. Moreover, if $\lambda=1$, $\sig_-$ is a horizontal plane (hence minimal and flat), and if $\lambda\neq1$, $\sig_-$ has positive Gauss-Kronecker curvature.	
		
	\end{itemize}	
	\label{Classification1}	
\end{teo}

For rotational $\hl$-surfaces in $\h2r$ non-intersecting the axis of rotation, we prove the following:

\begin{teo}\label{teoremafueraeje}
Let $\sig$ be a complete, rotational $\hl$-surface in $\h2r$ non-intersecting the rotation axis. Then, $\sig$ is properly immersed and diffeomorphic to $\S^1\times\R$. Moreover,
\begin{itemize}
\item[1.] If $\lambda>1/2$, then:
\begin{itemize}
\item[1.1.] either $\sig$ is the CMC cylinder $C_\lambda$ of radius $\arg\tanh\left(\frac{1}{2\lambda}\right)$, or
\item[1.2.] one end converges to $C_\lambda$ with the same asymptotic behavior as in item $\mathit{1}$ in Theorem \ref{Classification1}, and:
		\begin{itemize}		
		\item[a)] If $\lambda>\sqrt{5}/2$, the other end of $\sig$ has unbounded distance to the rotation axis and self-intersects infinitely many times.
		\item[b)] If $\lambda\leq \sqrt{5}/2$, the other end is a graph outside a compact set.
	\end{itemize}	
\end{itemize}
\item[2.] If $\lambda\leq 1/2$, then both ends are graphs outside compact sets.
\end{itemize}
	\label{Classification2}	
\end{teo}

\section{\large A phase plane analysis of rotational $\hh$-surfaces in $\h2r$}\label{sec2}
\vspace{-.5cm}

In the development of this section, we regard $\h2r$ as a submanifold of $\R^3_{-}\times\R$ endowed with the metric $+,+,-,+$. Consider a regular curve parametrized by arc-length $\alpha(s)=(x_1(s),0,x_3(s),z(s))\subset\h2r$, $x_1(s)>0$, $s\in I\subset\R$, contained in a vertical plane passing through the point $(0,0,1,0)$, and rotate it around the vertical axis $\{(0,0,1)\}\times\R$. Since $x_1^2(s)-x_3^2(s)=-1$, there exists a $C^1$ function $x(s)>0$ such that 
\[
\alpha(s)=(\sinh (x(s)),0,\cosh (x(s)),z(s)).
\]
For saving notation, we will simply denote $\alpha(s)$ by $(x(s),z(s))$. Now, note that the image of $\alpha(s)$ under this $1$-parameter group of rotations generates an immersed, rotational surface $\sig$ in $\h2r$ parametrized by 
\begin{equation}
\label{rotsurface}
\psi(s,\theta)=(\sinh(x(s))\cos\theta,\sinh(x(s))\sin\theta,\cosh (x(s)),z(s)):I\times(0,2\pi)\rightarrow\h2r,
\end{equation}
whose angle function at each point $\psi(s,\theta)$ is given by $\nu(\psi(s,\theta))=x'(s)$, $\forall s\in I$. Moreover, the principal curvatures on $\sig$ are
\begin{equation}
\label{princurv}
\kappa_1=\kappa_{\alpha}=x'z''-x''z',\hspace{.5cm} \kappa_2=\frac{z'}{\tanh x},
\end{equation}
where $\kappa_{\alpha}$ stands for the geodesic curvature of the profile curve $\alpha(s)$. Thus, the mean curvature of $\sig$ is easily related to the coordinates of $\alpha(s)$ and bearing in mind that $x'^2+z'^2=1$, we get that the function $x$ is a solution of the autonomous second order ODE
\begin{equation}\label{ODEx''}
x''=\frac{1-x'^2}{\tanh x}-2\varepsilon H_\sig\sqrt{1-x'^2},\hspace{.5cm} \varepsilon=\mathrm{sign}(z'),
\end{equation}
on every subinterval $J\subset I$ where $z'(s)\neq 0$ for all $s\in J$.

Now, assume that $\sig$ is an $\hh$-surface for some $\hh\in C^1([-1,1])$, that is, by using \eqref{PMCH2R} we get $H_{\sig}(\psi(s,\theta))=\hh(x'(s))$. Hence, after the change $y=x'$, we can rewrite \eqref{ODEx''} as the following autonomous ODE system 
\begin{equation}\label{1ordersys}
\left(\begin{array}{c}
x\\
y
\end{array}\right)'=\left(\begin{array}{c}
y\\
\displaystyle{\frac{1-y^2}{\tanh x}}-2\varepsilon\hh(y)\sqrt{1-y^2}
\end{array}\right)=:F_\varepsilon(x,y).
\end{equation}
At this point, we study this system by using a phase plane analysis as the first author did in \cite{BGM2}. Specifically, the \emph{phase plane} of \eqref{1ordersys} is defined as the half-strip $\Theta_\varepsilon:=(0,\infty)\times(-1,1)$, with coordinates $(x,y)$ denoting, respectively, the distance to the rotation axis and the angle function of $\sig$. The solutions $\gamma(s)=(x(s),y(s))$ of system \eqref{1ordersys} are called \emph{orbits}, and the \emph{equilibrium points} are the points $e_0^{\varepsilon}=(x_0^\varepsilon,y_0^\varepsilon)\in\Theta_\varepsilon$ such that $F_\varepsilon(x_0^\varepsilon,y_0^\varepsilon)=0$, which must lie in the axis $y=0$ according to system \eqref{1ordersys}.

Next, we compile some features that can be derived from the study of the phase plane. 

\begin{lem}\label{properties} 
In the above conditions, the following properties hold:
	\begin{enumerate}
		\item If $\varepsilon\hh(0)>0$, there is a unique equilibrium $e_0^{\varepsilon}=(x_0^\varepsilon,0)$ of \eqref{1ordersys} in $\Theta_\varepsilon$ given by 
		\[
		x_0^\varepsilon=\displaystyle{\arg\tanh\left(\frac{1}{2\varepsilon\hh(0)}\right)}.
		\]
		This equilibrium generates the right circular cylinder $\S^1(x_0^\varepsilon)\times\R$ of constant mean curvature $\hh(0)$ and vertical rulings.
		\item The Cauchy problem associated to \eqref{1ordersys} for the initial condition $(x_0,y_0)\in\Theta_\varepsilon$ has local existence and uniqueness. Thus, the orbits provide a foliation by regular $C^1$ curves of $\Theta_\varepsilon$ (or of $\Theta_\varepsilon-\{e_0^\varepsilon\}$, in case some $e_0^\varepsilon$ exists). 
		\item The instants $s_0\in J$ such that $\kappa_\alpha(s_0)=0$ are the ones for which $x''(s_0)=y'(s_0)=0$, i.e. those such that $(x(s_0),y(s_0))\in\Gamma_\varepsilon:=\Theta_\varepsilon\cap\{x=\Gamma_\varepsilon(y)\}$, where 
		\begin{equation}\label{gamma}
			\Gamma_\varepsilon(y)=\arg\tanh\left(\frac{\sqrt{1-y^2}}{2\varepsilon\hh(y)}\right).
		\end{equation}
		Moreover, the curve $\Gamma_{\varepsilon}$ is empty when $\varepsilon \hh(y)\leq0$.
		
		\item The axis $y=0$ and the curve $\Gamma_\varepsilon$ divide $\Theta_\varepsilon$ into connected components, called monotonicity regions, where the coordinates $x(s)$ and $y(s)$ of every orbit are strictly monotonous. Moreover, at each of these regions, we have
		\begin{equation}\label{signPC}
		\mathrm{sign}(\kappa_1)=\mathrm{sign}(-\varepsilon y'(s)), \quad 	\mathrm{sign}(\kappa_2)=\mathrm{sign}(\varepsilon).
		\end{equation}
	\end{enumerate}
\end{lem}

%Since $\alpha$ is arc-length parametrized, from Equation \eqref{princurv} we get $\kappa_\alpha(s)=x''(s)/\sqrt{1-x'(s)^2}$. Hence, the instants $s_0\in I$ such that $\kappa_\alpha(s_0)=0$ are the ones for which $x''(s_0)=y'(s_0)=0$. From 						\eqref{1ordersys} we conclude that they are located in $\Gamma_\varepsilon:=\Theta_\varepsilon\cap\{x=\Gamma_\varepsilon(y)\}$, where 
%		\begin{equation}\label{gamma}
%			\Gamma_\varepsilon(y)=\arg\tanh\left(\frac{\sqrt{1-y^2}}{2\varepsilon\hh(y)}\right).
%		\end{equation}
%		Moreover, the curve $\Gamma_{\varepsilon}$ is empty when $\varepsilon \hh(y)\leq0$.

Let us focus now on the behavior of the orbits of system \eqref{1ordersys} in more detail. Firstly, note that we can view such orbits as vertical graphs $y=y(x)$ where $y\neq 0$ (i.e. where $x'\neq 0$), so the chain rule yields
	\begin{equation}\label{orbitagrafo}
		y'(s)=\frac{dy}{ds}=\frac{dy}{dx}\frac{dx}{ds}=y'(x)y(x)=\frac{1-y(x)^2}{\tanh x}-2\varepsilon\hh(y(x))\sqrt{1-y(x)^2}.
	\end{equation}
Since at each monotonicity region the sign of the quantity $y(x)y'(x)$ is constant, the behavior of the orbit passing through some $(x_0,y_0)\in\te$ is determined by the signs of $y_0$ and $\Gamma_\varepsilon(y_0)-x_0$ (whenever $\Gamma_\varepsilon(y_0)$ exists). We detail it next:

\begin{lem}\label{monotonicity}
In the above conditions, the behavior of the orbit of \eqref{1ordersys} passing through a given point $(x_0,y_0)\in\Theta_\varepsilon$ such that $\Gamma_\varepsilon(y_0)$ exists is described as follows: 
\begin{enumerate}
	\item If $x_0>\Gamma_{\varepsilon}(y_0)$ (resp. $x_0<\Gamma_{\varepsilon}(y_0)$) and $y_0>0$, then $y(x)$ is strictly decreasing (resp. increasing) at $x_0$. 
	\item If $x_0>\Gamma_{\varepsilon}(y_0)$ (resp. $x_0<\Gamma_{\varepsilon}(y_0)$) and $y_0<0$, then $y(x)$ is strictly increasing (resp. decreasing) at $x_0$.
	\item If $y_0=0$, then the orbit passing through $(x_0,0)$ is orthogonal to the $x$ axis.
	\item If $x_0=\Gamma_{\varepsilon}(y_0)$, then $y'(x_0)=0$ and $y(x)$ has a local extremum at $x_0$.
\end{enumerate}	
\end{lem}

The following result discusses how an orbit $\gamma(s)=(x(s),y(s))\in\te$ behaves when $x(s)\rightarrow\infty$.
\begin{prop}\label{orbitaescapa}
Let $\gamma(s)=(x(s),y(s))$ be an orbit in $\te$ such that $(x(s),y(s))\rightarrow(\infty,y_0)$  when $|s|\rightarrow\infty$, with $y_0\in(-1,1)$. Then, $2\varepsilon\hh(y_0)=\sqrt{1-y_0^2}$.
\end{prop}
\begin{proof}
Suppose that $\gamma(s)\in\te$ satisfies $x(s)\rightarrow\infty$ and $y(s)\rightarrow y_0\in(-1,1)$ when $|s|\rightarrow\infty$. Then, there exists $s_0\in\R$ such that for every $s\in\R$ satisfying $|s|>|s_0|$, $\gamma(s)$ is strictly contained in some monotonicity region and so does not intersect the axis $y=0$. Thus, $\gamma(s)$ can be written as a vertical graph $y=y(x)$ satisfying $y(x)\rightarrow y_0$ and $y'(x)\rightarrow 0$ as $x\rightarrow\infty$, and substituting into Equation \eqref{orbitagrafo} we get 
\[1-y_0^2-2\varepsilon\hh(y_0)\sqrt{1-y_0^2}=0
\] 
concluding the result.
\end{proof}

%Escribir frase para introducir figura y pensar en figura. Hay que cambiar dentro de la figura el -1 de Gamma y Theta.
% 
%\begin{figure}[H]
%	\centering
%	\includegraphics[width=0.5\textwidth]{fasesh2r.pdf}  
%	\caption{An example of the phase plane $\Theta_1$ in which the curve $\Gamma_1$ has an asymptote at $y_0$. The point $e_0^1$ is the equilibrium, there are five monotonicity regions and the arrows show how an orbit behave in each of them.}
%	\label{fasesh2r}
%\end{figure}

Additionally, we highlight that the possible endpoints of an orbit are restricted as shown in \cite[Theorem 4.1, pp. 13-14]{BGM2}.
\begin{lem}
	No orbit in $\Theta_\varepsilon$ can converge to some point	of the form $(0,y)$ with $|y|<1$.
	\label{endpoints}
\end{lem}

From this result, we conclude that in case that an orbit converges to the axis $x=0$, it must do it to the points $(0,\pm 1)$.  However, we point out that the existence of an orbit with endpoint at $(0,\pm 1)$ can not be guaranteed by solving the Cauchy problem since system \eqref{1ordersys} has a singularity at the points with $x_0=0$. In this case, we can ensure the existence of such an orbit by using the work of Gálvez and Mira \cite{GaMi2} in which they solved the Dirichlet problem for radial solutions of an arbitrary \emph{fully nonlinear elliptic PDE}. Therefore, it is ensured the existence of rotational $\hh$-surfaces in $\h2r$ intersecting the rotation axis in an orthogonal way (see Section $3$ in \cite{Bue5} for further details). Furthermore, this fact derives the following result in the phase plane $\Theta_\varepsilon$.

\begin{cor}\label{ejefase}
Let be $\varepsilon,\delta\in\{-1,1\}$ such that $\varepsilon\hh(\delta)>0$ and denote $\sigma=\mathrm{sign}(\delta)$. Then, there exists a unique orbit $\gamma_\sigma$ in $\Theta_\varepsilon$ that has $(0,\delta)\in \overline{\Theta_\varepsilon}$ as an endpoint. There is no such an orbit in $\Theta_{-\varepsilon}$.
\end{cor}
The unique rotational $\hh$-surface associated to the orbit $\gamma_\sigma$, that intersects orthogonally the axis of rotation at some point having unit normal $\delta\partial_z$, will be denoted by $\sig_\sigma$.

\section{\large Existence of $\hh$-bowls and $\hh$-catenoids in $\h2r$}\label{sec3}
\vspace{-.5cm}

This section is devoted to construct examples of properly embedded rotational $\hh$-surfaces in $\h2r$, under some additional assumptions over the function $\hh\in C^1([-1,1])$. For this purpose, we follow the ideas compiled in Section 3 in \cite{BGM2}.

Firstly, we show the existence of entire, strictly convex $\hh$-graphs in $\h2r$. Recall that there exists an entire graph with CMC equal to $H_0\in\R$ if and only if $|H_0|\leq 1/2$. In addition, there exists a sphere with CMC equal to $H_0$ if and only if $|H_0|>1/2$.

Regarding $\hh$-surfaces, a necessary and sufficient condition to ensure the existence of an $\hh$-sphere in $\h2r$ is that $\hh$ must be an even function satisfying 
\begin{equation}\label{ineq}
2|\hh(y)|>\sqrt{1-y^2}
\end{equation}
for every $y\in[-1,1]$ (see Proposition $3.6$ and Theorem $4.1$ in \cite{Bue5}). Furthermore, under such hypotheses over $\hh$, we must take into account that the existence of an $\hh$-sphere in $\h2r$ forbids the existence of entire vertical $\hh$-graphs in $\h2r$, since this fact would yield to a contradiction with the maximum principle. Therefore, we construct the announced $\hh$-graphs by assuming the failure of inequality \eqref{ineq}.

\begin{prop}\label{bowls}
Let $\hh$ be a $C^1$ function on $[-1,1]$, and suppose that there exists $y_*\in [0,1]$ (resp. $y_*\in [-1,0]$) such that $2\varepsilon\hh(y_*)=\sqrt{1-{y_*}^2}$. Then, there exists an upwards-oriented (resp. downwards-oriented) entire rotational $\hh$-graph $\sig$ in $\h2r$. Moreover: 
\vspace{-.2cm}
\begin{itemize}
	\item[1.] either $\sig$ is a horizontal plane, 
	\item[2.] or $\sig$ is a strictly convex graph.
\end{itemize} 
\end{prop}

\begin{proof}
If $y_*=1$ (resp. $y_*=-1$), then $\hh(1)=0$ (resp. $\hh(-1)=0$), and we choose the surface $\sig$ as the horizontal plane $\mathbb{H}^2\times\{t_0\}$, $t_0\in\R$, which is minimal. Then, by considering $\sig$ with upwards orientation (resp. downwards orientation), its angle function is $\nu\equiv 1$ (resp. $\nu\equiv-1$), and so from \eqref{PMCH2R} the result is trivial.

Now, suppose that $2\varepsilon\hh(y_*)=\sqrt{1-{y_*}^2}$ holds for some $y_*\in[0,1)$, assume $\hh(1)>0$ and define $y_0:=\max\{y\in[0,1):\ 2\hh(y)=\sqrt{1-y^2}\}$. Note that $y_0$ is well defined since $\hh(1)>0$, and by continuity $2\hh(y)>\sqrt{1-y^2}$ if $y\in(y_0,1]$. Hence, the horizontal graph $\Gamma_1=\Theta_1\cap\Gamma_1(y)$ defined by \eqref{gamma} has one connected component when we restrict $\Gamma_1(y)$ to $(y_0,1]$. Moreover, it satisfies $\Gamma_1(1)=0$ and $\Gamma_1(y)\rightarrow \infty$ as $y\rightarrow y_0$.

Let us take $\Lambda:=\{(x,y)\in\Theta_1:y>y_0\}$, and define $\Lambda_+:=\{(x,y)\in\Lambda: x>\Gamma_1(y)\}$ and $\Lambda_-:=\{(x,y)\in\Lambda: x<\Gamma_1(y)\}$. Note that $\Lambda\setminus\Gamma_1$ is divided into two connected components $\Lambda_+$ and $\Lambda_-$, which are precisely monotonicity regions of $\Theta_1$ because of item $4$ in Lemma \ref{properties}. Indeed, by Lemma \ref{monotonicity} each orbit $y=y(x)$ in $\Lambda_+$ (resp. $\Lambda_-$) satisfies that $y'(x)<0$ (resp. $y'(x)>0$).

Now, from Corollary \ref{ejefase} it is known that there exists a unique orbit $\gamma_+$ in $\Theta_1$ with $(0,1)$ as an endpoint. Additionaly, by the aforementioned monotonicity properties and Lemma \ref{monotonicity} it is clear that $\gamma_+$ is globally contained in $\Lambda_+$. Thus, $\gamma_+$ can be globally defined by a graph $y=f(x)$, where $f\in C^1([0,\infty))$ satisfies $f(0)=1$, $f(x)\rightarrow y_1\geq y_0$ as $x\rightarrow\infty$, and $f'(x)<0$ for all $x>0$. As a matter of fact, Proposition \ref{orbitaescapa} ensures us that $y_1=y_0$ (see Figure \ref{bowlfases}, left). 

Consequently, the $\hh$-surface $\sig_+$ generated by the orbit $\gamma_+$ is an entire rotational graph in $\h2r$. It remains to prove that $\sig_+$ is strictly convex. On the one hand, since $\gamma_+$ is totally contained in $\Theta_1$, hence $\varepsilon=1$, we deduce by \eqref{signPC} that $\kappa_2$ of $\sig_+$ is everywhere positive. On the other hand, $\gamma_+$ is totally contained in $\Lambda_+$, so \eqref{signPC} implies that $\kappa_1$ of $\sig_+$ is also everywhere positive concluding the proof of this case.

\begin{figure}[H]
	\centering
	\includegraphics[width=0.8\textwidth]{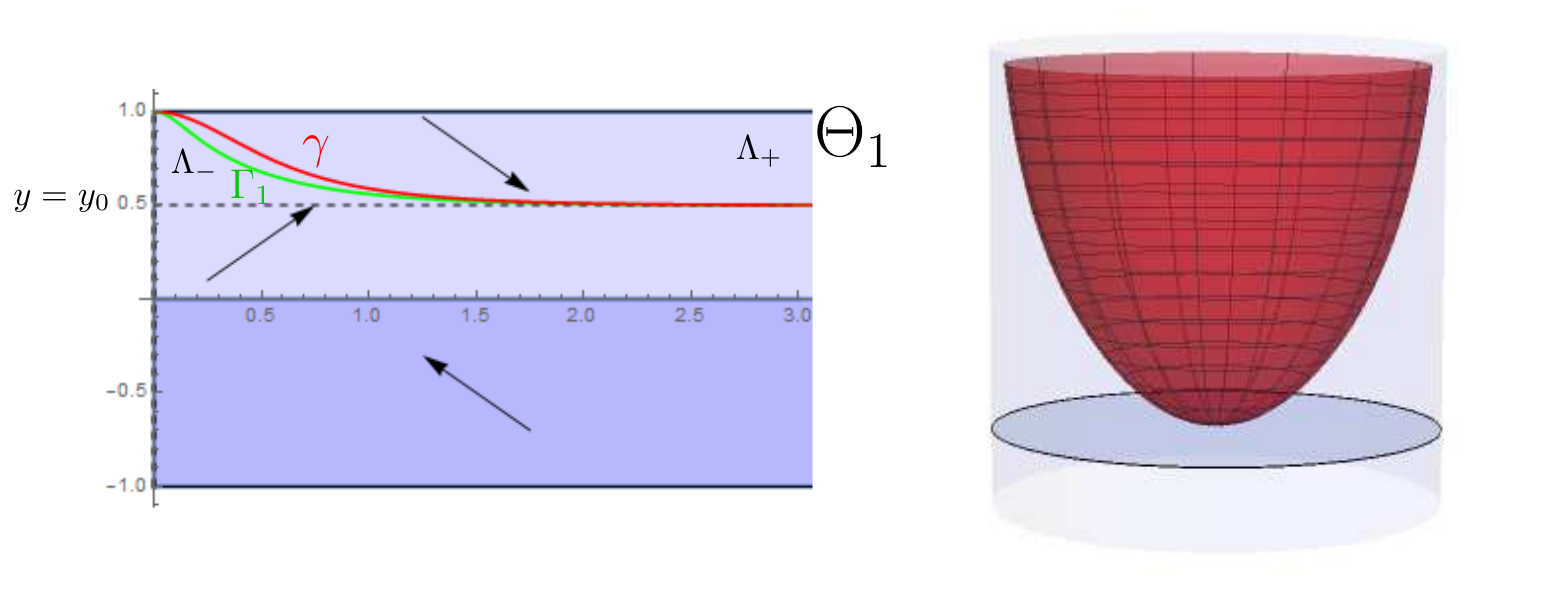}
	\caption{Left: the phase plane $\Theta_1$, the regions $\Lambda_+$ and $\Lambda_-$, the curve $\Gamma_1$ in green and the orbit $\gamma_+$ in red. Right: an $\hh$-bowl in $\h2r$. The prescribed function is $\hh(y)=\sqrt{3}(y-0.25)$.}
	\label{bowlfases}
\end{figure}

To finish, note that the case $\hh(-1)>0$, $y_*\leq0$ is treated analogously; and the two remaining cases, $\hh(1)<0$, $y_*\geq0$ and $\hh(-1)<0$, $y_*\leq0$, can be reduced to the previous ones by changing the orientation.
\end{proof}

These $\hh$-surfaces will be called $\hh$\emph{-bowls}, in analogy with the theory of self-translating solitons of the mean curvature flow (see \cite{Bue1, Bue2, LiMa}) that we extend with the previous result. See Figure \ref{bowlfases}, right, for a graphic of an $\hh$-bowl in $\h2r$.

Secondly, we study the existence of \emph{catenoid-type} rotational $\hh$-surfaces under appropiate conditions for the prescribed function $\hh$.

\begin{prop}\label{hcats}
Let $\hh$ be a  $C^1$  function on $[-1,1]$, and suppose that $\hh\leq0$ and $\hh(\pm1)=0$. Then, there exists a one-parameter family of properly embedded, rotational $\hh$-surfaces in $\h2r$ of strictly negative extrinsic curvature at every point, and diffeomorphic to $\S^1\times \R$. Each example is a bi-graph over $\mathbb{H}^2-\D_{\mathbb{H}^2}(x_0)$, where $\D_{\mathbb{H}^2}(x_0)=\{x\in\mathbb{H}^2: |x|_{\mathbb{H}^2}<x_0\}$, for some $x_0>0$.
\end{prop}

\begin{proof}
Let $\sig(x_0)$ be the rotational $\hh$-surface generated by the arc-length parametrized curve $\alpha(s)=(x(s),z(s))$ with the following initial conditions
\[
x(0)=x_0,\; z(0)=0, \;\mathrm{and}\; z'(0)=1,\ \mathrm{with}\ x_0>0.
\]
Then, the orbit $\gamma(s)=(x(s),y(s))$ of system \eqref{1ordersys} associated to $\alpha(s)$ passes through the point $(x_0,0)$ at $s=0$. Moreover, $\gamma$ is contained in $\Theta_1$ around such a point, that is, $\varepsilon=1$. 

Observe that the curve $\Gamma_1$ given by \eqref{gamma} does not exist because of the assumption $\hh\leq 0$. Consequently, by item $4$ in Lemma \ref{properties} there are two monotonicity regions of $\Theta_1$ given by $\Lambda_+:=\{(x,y)\in\Theta_1: y>0\}$ and $\Lambda_-:=\{(x,y)\in\Theta_1:y<0\}$. Then, from \eqref{1ordersys} we know that $\gamma$ satisfies $x'>0$ and $y'>0$ in $\Lambda_+$, and $x'<0$ and $y'>0$  in $\Lambda_-$; see Figure \ref{CompletoCatenoide}, left.

\begin{figure}[h]
	\begin{center}
		\includegraphics[width=.75\textwidth]{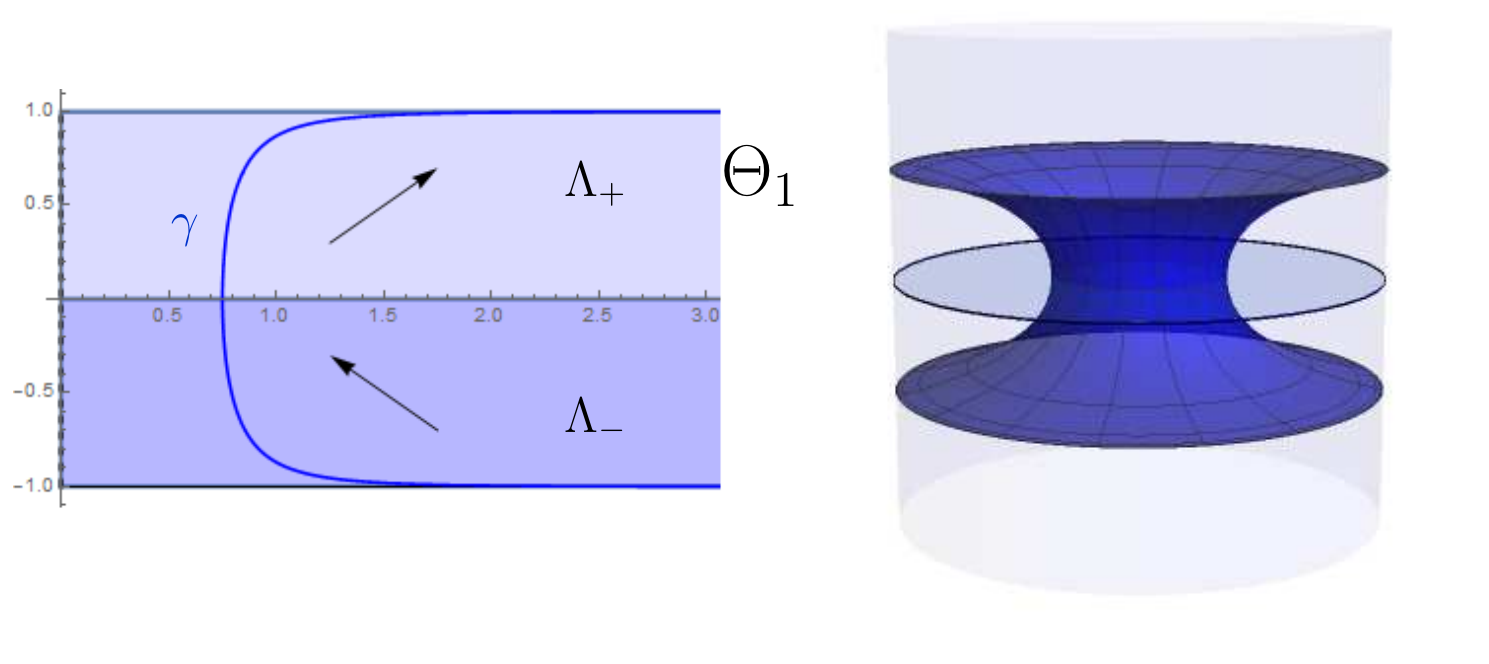}
	\end{center}
	\caption{Left: the phase plane $\Theta_1$, the monotonicity regions $\Lambda_+$ and $\Lambda_-$, and an orbit $\gamma$ in blue. Right: an $\hh$-catenoid in $\h2r$. The prescribed function is $\hh(y)=y^2-1$.}
	\label{CompletoCatenoide}
\end{figure}
Let us prove now that $\gamma$ must be a proper arc strictly contained in $\t1$ satisfying $\gamma(s)\rightarrow(\infty,\pm1)$ as $s\rightarrow\pm\infty$. First note that, the assumption $\hh\leq0$ implies that given $y_0\in(-1,1)$, the equation $2\hh(y_0)=\sqrt{1-y_0^2}$ has no solutions, so from Proposition \ref{orbitaescapa} $\gamma$ cannot satisfy $\gamma(s)\rightarrow(\infty,y_0)$ as $s\rightarrow\pm\infty$. Second, since $\hh\in C^1$ and $\hh(\pm1)=0$, from the uniqueness of the Cauchy problem associated to \eqref{1ordersys}, the curve $(s,\pm1),\ s>0,$ is a solution to \eqref{1ordersys} which corresponds to a horizontal plane $\mathbb{H}^2\times\{t_0\},\ t_0\in\R,$ endowed with $\pm\partial_z$ as unit normal, that is, $\gamma(s)$ cannot satisfy $\gamma(s_0)=(x_0,\pm1)$ for some $s_0\in\R$. Finally, it remains to show that $\gamma(s)$ cannot converge to some $(x_0,\pm1),\ x_0>0,$ as $s\rightarrow\pm\infty$. Otherwise, there would exist $s_0\in\R$ such that for $|s|>|s_0|$, $x(s)$ is a monotonous function satisfying $x(s)\rightarrow x_0$ as $|s|\rightarrow\infty$. Thus, the mean value theorem ensures us that $x'(s)=y(s)\rightarrow0$, which is a contradiction with the fact that $y(s)\rightarrow\pm1$.

%As a consequence, if $\hh(\pm1)=0$ the phase plane $\Theta_\varepsilon$ extends to the boundary lines $(x,\pm1),\ x>0$. Therefore, an orbit $\gamma(s)$ cannot satisfy $\gamma(s_0)=(x_0,\pm1)$ for some $s_0\in\R$, since it would contradict the uniqueness of the Cauchy problem. 

%By monotonicity and the aforementioned discussions, the only possibility for $\gamma$ is to satisfy $\gamma(s)\rightarrow(\infty,\pm1)$ as $s\rightarrow\pm\infty$.

Thus, $\sig(x_0)$ is a bi-graph in $\h2r$ over $\Omega(x_0):=\mathbb{H}^2-\D_{\mathbb{H}^2}(x_0)$, with the topology of $\S^{1}\times\R$. Indeed, $\sig(x_0)=\sig_1\cup\sig_2$ where both $\sig_i$ are graphs over $\Omega(x_0)$ with $\partial\sig_i=\partial\Omega(x_0)$ and $\sig_i$ meets the horizontal plane $\mathbb{H}^2\times\{0\}$ in an orthogonal way along $\partial\sig_i$ (see Figure \ref{CompletoCatenoide} right). 

It remains to prove that the extrinsic curvature of $\sig(x_0)$ is strictly negative. By \eqref{1ordersys} we get $y'(s)>0$ for all $s$, so from \eqref{signPC} we derive $\kappa_1<0$ and $\kappa_2>0$ at every $p\in\sig$.
\end{proof}

This one-parameter family of rotational $\hh$-surfaces is a generalization of the usual minimal catenoids in $\h2r$, and this is the reason for calling them $\hh$\emph{-catenoids}. As happens for the minimal catenoids, the $\hh$-catenoids are parametrized by their \emph{necksizes}, i.e. the distance of their \emph{waists} to the axis of rotation.

\section{\large Classification of rotational $\hh$-surfaces with linear prescribed mean curvature}\label{sec4}
\vspace{-.5cm}

Our aim in this section is to classify the rotational examples of the following class of $\hh$-surfaces:

\begin{defi}
An oriented surface $\sig$ immersed in $\h2r$ is an $\hl$-surface if its mean curvature function $H_\sig$ satisfies
	\begin{equation}\label{hlambda}
	H_\sig(p)=\hh_\lambda(\nu(p))=a\nu(p)+\lambda,\ \forall p\in\sig,\hspace{.5cm} a,\lambda\in\R.
	\end{equation}
\end{defi}
Note that if $a=0$, then we are studying surfaces with constant mean curvature equal to $\lambda$. Also, if $\lambda=0$ the $\hl$-surfaces are translating solitons of the mean curvature flow, see \cite{Bue1, Bue2, LiMa}. Hence, we suppose that $a,\lambda$ are not null in order to avoid these cases. After a homothety in $\h2r$ we can suppose $a=1$ in Equation \eqref{hlambda}. Moreover, if $\sig$ is an $\hl$-surface, then $\sig$ with its opposite orientation is an $\hh_{-\lambda}$-surface. Therefore, we will assume $\lambda>0$ without losing generality. In particular this implies that $\varepsilon\hh(0)>0$ if and only if $\varepsilon=1$, and consequently the equilibrium $e_0=(\arg\tanh(\frac{1}{2\varepsilon\hh(0)}),0)$ can only exist in $\t1$. The CMC vertical cylinder generated by $e_0$ will be denoted by $C_\lambda$.

First, we announce two technical results that will be useful in the sequel. The first one was originally proved by López in \cite{Lop} for surfaces in $\R^3$ whose mean curvature is given by Equation \eqref{hlambda}.
\begin{lem}\label{noclosed}
There do not exist closed $\hl$-surfaces in $\h2r$.
\end{lem}

\begin{proof}
Arguing by contradiction, suppose that $\sig$ is a closed $\hl$-surface in $\h2r$. If $h:\sig\rightarrow\R$ denotes the \emph{height function} of $\sig$, it is known that the Laplace-Beltrami operator $\Delta_\sig$ of $h$ is $\Delta_ \sig h=2H_\sig\langle\eta,\partial_z\rangle$. Since $\sig$ is an $\hl$-surface, we get
$$
\Delta_\sig h=2\langle\eta,\partial_z\rangle^2+2\lambda\langle\eta,\partial_z\rangle.
$$
We integrate this equation in $\sig$. By the divergence theorem and since $\partial\sig=\varnothing$, we have
$$
0=\int_\sig\langle\eta,\partial_z\rangle^2d\sig+\lambda\int_\sig\langle\eta,\partial_z\rangle d\sig.
$$
The second integral is zero by the divergence theorem, since the constant vector field $\partial_z$ has zero divergence. So, the first integral vanishes, that is, $\sig$ is contained in a cylindrical surface of the form $\beta\times\R,\ \beta\subset\mathbb{H}^2$ being a curve, which contradicts that $\sig$ is compact.
\end{proof}

The second result forbids the existence of closed orbits in the phase plane of system \eqref{1ordersys} for some prescribed functions $\hh$. It follows from Bendixson-Dulac theorem, a classical result which appears in most textbooks on differential equations; see e.g. \cite{ADL}. 

\begin{teo}\label{noorbitascerradas}
Let $\hh$ be a  $C^1$  function on $[-1,1]$ such that $\hh'(y)\neq0,\ \forall y\in(-1,1)$. Then, there do not exist closed orbits in $\te$.
\end{teo}

\begin{proof}
Let us write system \eqref{1ordersys} as
$$
\left(\begin{array}{c}
x\\
y
\end{array}\right)'=\left(\begin{array}{c}
y\\
\displaystyle{\frac{1-y^2}{\tanh x}}-2\varepsilon\hh(y)\sqrt{1-y^2}
\end{array}\right)=\left(\begin{array}{c}
P(x,y)\\
Q(x,y)
\end{array}\right),
$$
and define the function $\alpha:\te\rightarrow\R$ and the vector field $V:\te\rightarrow\te$ as
$$
\alpha(x,y)=\frac{\sinh x}{\sqrt{1-y^2}},\hspace{.5cm} V(x,y)=\alpha(x,y)(P(x,y),Q(x,y)).
$$

Arguing by contradiction, suppose that there exists some closed orbit $\overline{\gamma}$ in $\te$ and name $\Omega$ to its inner region. A simple computation yields $\mathrm{div} V=(\alpha P)_x+(\alpha Q)_y=-2\varepsilon\hh'(y)\sinh x$, which has constant sign since $x>0$ in $\te$ and $\hh'(y)\neq0$. Therefore, the divergence theorem in $\Omega$ yields
$$
0\neq\int_\Omega\mathrm{div} V=\int_\gamma\langle V,\textbf{n}_\gamma\rangle=0,
$$
where $\textbf{n}_\gamma$ is the unit normal to the curve $\gamma$. Recall that the last integral vanishes since $V$ is everywhere tangent to $\gamma$. This contradiction proves the result.
\end{proof}
In particular, the prescribed function $\hl(y)=y+\lambda,\ \lambda>0$ lies in the hypothesis of Theorem \ref{noorbitascerradas}, hence the phase plane $\te$ of system \eqref{1ordersys} for $\hl$-surfaces does not have closed orbits.

Now, suppose that $\sig$ is a rotational $\hl$-surface generated by an arc-length parametrized curve $\alpha(s)=(x(s),z(s))$. Then, \eqref{PMCH2R} and \eqref{princurv} yields 
\begin{equation}\label{CurvMed}
2H_\sig=2(x'+\lambda)=x'z''-x''z'+\frac{z'}{\tanh x}.
\end{equation}  

Our first goal is to study the structure of the orbits around $e_0=(\arg\tanh(\frac{1}{2\lambda}),0)$. Recall that $e_0$ exists if and only if $\lambda>1/2$; otherwise, $\arg\tanh$ is not well defined. The \emph{linearized} system of \eqref{1ordersys} at $e_0$ is given by
\begin{equation}%\label{linearized}
\displaystyle{\left(\begin{matrix}
0&1\\
1-4\lambda^2& -2
\end{matrix}\right)},
\end{equation}
whose eigenvalues are 
$$
\displaystyle{\mu_1=-1+\sqrt{2-4\lambda^2},\hspace{1cm}\text{and} \hspace{1cm} \mu_2=-1-\sqrt{2-4\lambda^2}}.
$$ 
From standard theory of non-linear autonomous systems we derive:
\begin{lem}\label{comportamientoequilibrio} In the above conditions, we have:
\begin{itemize}
	\item If $\displaystyle{\lambda>\frac{\sqrt{2}}{2}}$, then $\mu_1$ and $\mu_2$ are complex conjugate with negative real part. Thus, $e_0$ has an \emph{inward spiral} structure, and every orbit close enough to $e_0$ converges asymptotically to it spiraling around infinitely many times.
	
	\item If $\displaystyle{\lambda=\frac{\sqrt{2}}{2}}$, then $\mu_1=\mu_2=-1$. Thus, $e_0$ is an asymptotically stable improper node, and every orbit close enough to $e_0$ converges asymptotically to it, maybe spiraling around a finite number of times.
	
	\item If $\displaystyle{0<\lambda<\frac{\sqrt{2}}{2}}$, then $\mu_1$ and $\mu_2$ are different and real. Thus, $e_0$ is an asymptotically stable node and has a \emph{sink} structure, hence every orbit close enough to $e_0$ converges asymptotically to it \emph{directly}, i.e. without spiraling around.
\end{itemize}
\end{lem}

Now we stand in position to prove Theorem \ref{teoremasaleeje}.

\begin{prooft1}
Note that the behavior of the orbits in each phase plane $\Theta_\varepsilon$ depends on
 the curve $\Gamma_\varepsilon$ and the monotonicity regions generated by it. Consequently, we analyze three different cases for $\lambda$: $\lambda>\sqrt{5}/2$, $\lambda=\sqrt{5}/2$ and $0<\lambda<\sqrt{5}/2$.

{\underline{Case $\lambda>\sqrt{5}/2$}}

Let us assume $\lambda>\sqrt{5}/2$. For $\varepsilon=1$, the curve $\Gamma_1$ given by \eqref{gamma} is a compact connected arc in $\Theta_1$ joining $(0,1)$ and $(0,-1)$, whereas for $\varepsilon=-1$, since $\hh$ is positive, the curve $\Gamma_{-1}$ does not exist in $\Theta_{-1}$. As a matter of fact, by item $4$ in Lemma \ref{properties} we know that there are four monotonicity regions in $\Theta_1$ which will be denoted by $\Lambda_1,\dots,\Lambda_4$, and there are only two monotonicity regions in $\Theta_{-1}$ which will be denoted by $\Lambda_+$ and $\Lambda_-$. See Figure \ref{PlanosFasesCaso1}. 

\begin{figure}[H]
	\centering
	\includegraphics[width=.8\textwidth]{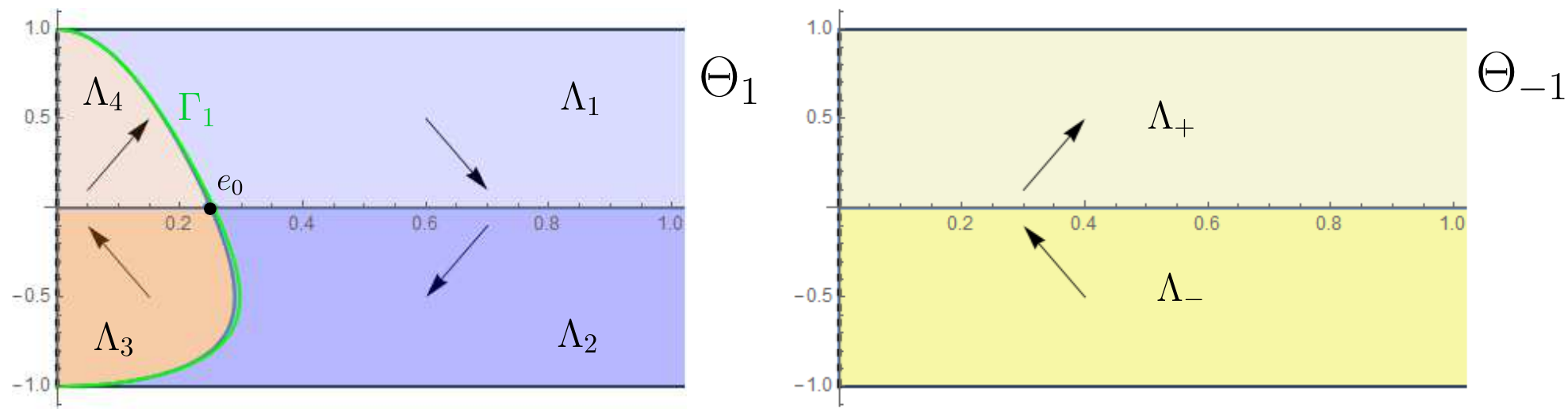}
	\caption{The phase planes $\Theta_1$ and $\Theta_{-1}$ for $\lambda>\sqrt{5}/2$ with their monotonicity regions and the direction of the motion of the orbits at each of them.}
	\label{PlanosFasesCaso1}
\end{figure}

Now, by using Corollary \ref{ejefase} it is clear that there exists a unique orbit $\gamma_+$ (resp. $\gamma_-$) in $\Theta_1$ with $(0,1)$ (resp. $(0,-1)$) as an endpoint.

On the one hand, let us study the behavior of $\gamma_+$. Firstly, we can suppose that such an orbit satisfies $\gamma_+(0)=(0,1)$ in $\Theta_1$, i.e., it generates an arc-length parametrized curve $\alpha_+(s)$ intersecting orthogonally the rotation axis with upwards oriented unit normal at $s=0$ . Because of the monotonicity properties, $\gamma_+(s)$ is strictly contained in the region $\Lambda_1$ for $s>0$ small enough. However, the orbit $\gamma_+$ cannot stay forever in $\Lambda_1$, otherwise $\gamma_+$ would be globally defined by a graph $y=f(x)$ such that 
\[
f(0)=1,\quad f'(x)<0\;\; \forall x>0\quad  \textrm{and}\quad \lim\limits_{x\rightarrow\infty}f(x)=c\in[0,1).
\]
This contradicts Proposition \ref{orbitaescapa} since $\lambda>\sqrt{5}/2$ and hence $2\hl(y)=2(y+\lambda)>\sqrt{1-y^2}$. Thus, $\gamma_+(s)$ intersects the axis $y=0$ in an orthogonal way at a point $(x_+,0)$ with $x_+>\arg\tanh\left(\frac{1}{2\lambda}\right)$ at some finite instant $s_+>0$.

On the other hand, for the orbit $\gamma_-$ we assume that $\gamma_-(0)=(0,-1)$ in $\Theta_1$, that is, it generates an arc-length parametrized curve $\alpha_-(s)$ intersecting orthogonally the axis of rotation with downwards oriented unit normal at $s=0$. By an analogous reasoning, we can assert that $\gamma_-(s)$ intersects $y=0$ orthogonally at $(x_-,0)$ with $x_->\arg\tanh\left(\frac{1}{2\lambda}\right)$ at some finite instant $s_-<0$.

Now, we prove that $x_+<x_-$ arguing by contradiction. First, see that if $x_+=x_-=\bar{x}$, by uniqueness of the Cauchy problem the orbits $\gamma_+$ and $\gamma_-$ could be smoothly glued together constructing a larger orbit $\bar\gamma$ which would be a compact arc joining the points $(0,1)$, $(\bar{x},0)$ and $(0,-1)$ and so the rotational $\hl$-surface generated would be a rotational $\hh$-sphere, which is impossible because of Lemma \ref{noclosed}. Additionally, if $x_+>x_-$, it would mean that $\gamma_-$ intersect the axis $y=0$ at the left-hand side of $\gamma_+$, and consequently the only possibility for $\gamma_-$ would be to enter the region $\Lambda_1$, later $\Lambda_4$ and after that $\Lambda_3$. In any case, $\gamma_-$ cannot converge to any point $(0,y),\ |y|<1$ in virtue of Lemma \ref{endpoints}. Repeating this process, and since $\gamma_-$ cannot self-intersect nor converge to a closed orbit by Theorem \ref{noorbitascerradas}, $\gamma_-$ finishes converging asymptotically to the equilibrium $e_0$ as $s\rightarrow-\infty$, spiriling around infinitely many times. This is a contradiction with the inward spiral structure of $e_0$, since this orbit would tend to escape from $e_0$ when $s$ increases. See Figure \ref{PlanosFasesCaso1orb}.

\begin{figure}[H]
	\centering
	\includegraphics[width=.4\textwidth]{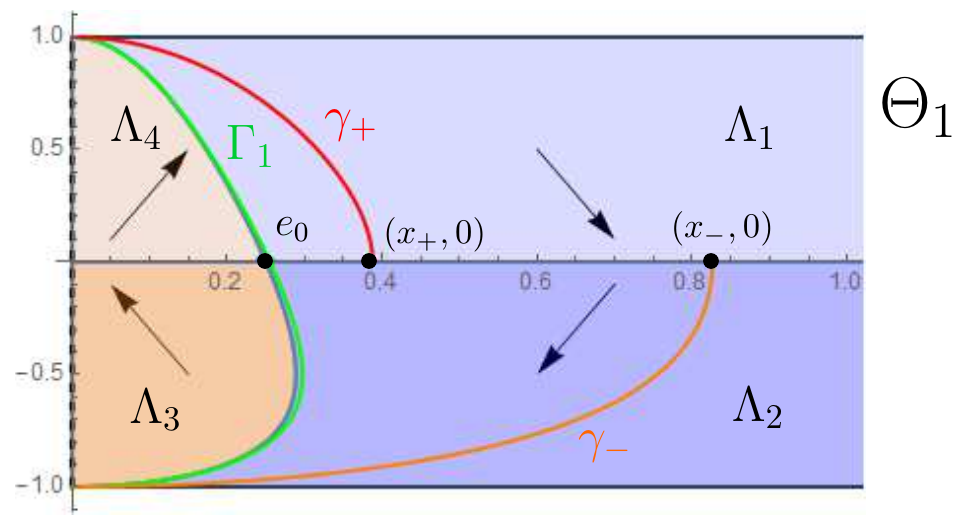}
	\caption{The phase plane $\Theta_1$ for $\lambda>\sqrt{5/2}$ with the configurarion of the orbits $\gamma_+$ and $\gamma_-$, plotted in red and orange, respectively, until they intersect the axis $y=0$.}
	\label{PlanosFasesCaso1orb}
\end{figure}
 
Let us continue by analyzing the global behavior of both orbits. Firstly, when $\gamma_+$ passes through $(x_+,0)$, it enters to $\Lambda_2$ but cannot intersect $\gamma_-$, so $\gamma_+$ has to enter to $\Lambda_3$. After that, due to the monotonicity properties and Lemma \ref{endpoints} we deduce that $\gamma_+$ has to enter to $\Lambda_4$ and intersect $\Gamma_1$. Since $\gamma_+$ cannot self-intersect nor converge to a limit closed orbit in $\t1$ in virtue of Theorem \ref{noorbitascerradas}, the only possibility for $\gamma_+$ is to repeat this behavior and eventually converge asymptotically to the equilibrium $e_0$ (see the plot of $\gamma_+$ in Figure \ref{PlanosFasesCaso1orb2} top-left). Furthermore, since $\lambda>\sqrt{2}/2$, $\gamma_+$ spirals around $e_0$ infinitely many times.

In this way, $\gamma_+$ generates the curve $\alpha_+(s)=(x_+(s),z_+(s))$ satisfying: the $x_+(s)-$coordinate is bounded by the value $x_+$ and converges to $\arg\tanh\left(\frac{1}{2\lambda}\right)$; and the $z_+(s)$-coordinate is strictly increasing since $\gamma_+\subset\Theta_1$, hence $z'_+(s)>0$. Then, $\alpha_+(s)$ is an embedded curve that converges to the line $x=\arg\tanh\left(\frac{1}{2\lambda}\right)$ intersecting it infinitely many times. Therefore, after rotating such a curve around the rotation axis, we derive that the generated surface $\Sigma_+$ is a properly embedded, simply connected $\hl$-surface that converges to the CMC cylinder $C_\lambda$ intersecting it infinitely many times. See $\Sigma_+$ in Figure \ref{PlanosFasesCaso1orb2} top-right.

Now we focus on $\gamma_-$, which intersects the axis $y=0$ at the point $\gamma_-(s_-)=(x_-,0)$ at some finite instant $s_-<0$. So, when the parameter $s<s_-$ decreases, $\gamma_-$ enters to $\Lambda_1$. Bearing in mind that $\gamma_+$ and $\gamma_-$ cannot intersect, from the monotonicity properties we deduce that $\gamma_-$ has as endpoint $\gamma_-(s_1)=(x_1,1)$ with $0<x_1<x_-$ and $s_1<s_-$. Then, $\gamma_-$ generates the curve $\alpha_-(s)=(x_-(s),z_-(s))$ satisfying that: $x_-(s_1)=x_1$ and $x_-'(s_1)=1$, so from \eqref{CurvMed} we get $z_-''(s_1)>0$, i.e., the height of $\alpha_-$ reaches a minimum at $s_1$. If we name $\sig_-$ to the $\hl$-surface associated to $\gamma_-$ and generated by rotating $\alpha_-(s)$, the image of the points $\alpha_-(s_1)$ under such rotation corresponds to points on the boundary of $\sig_-$ having unit normal $\partial_z$.

Therefore, for $s<s_1$, $s$ close enough to $s_1$, the height function $z_-(s)$ of $\alpha_-(s)$ is decreasing, i.e. $z_-'(s)<0$. Consequently, $\alpha_-(s)$ for $s<s_1$ close enough to $s_1$ generates an orbit in $\Theta_{-1}$ (since $\varepsilon=\mathrm{sign}(z'_-)=-1$), which will be named $\gamma_-$ for saving notation. Hence, the orbit $\gamma_-(s)$ \emph{continues} from $\Theta_1$ to $\Theta_{-1}$ as $s$ decreases from $s=s_1$. At this point, the continuation of $\gamma_-$ between the phase planes $\Theta_{\pm1}$ has to be understood as the extension of $\sig_-$ by solving the Cauchy problem for rotational vertical $\hl$-graphs having the same vertical unit normal. 

Hence, $\gamma_-(s)$ belongs to $\Theta_{-1}$ for $s<s_1$ close enough to $s_1$ and lies in the region $\Lambda_+$. Once again, by monotonicity, $\gamma_-$ must intersect the axis $y=0$ orthogonally and enter to $\Lambda_-$. As $\gamma_-$ cannot stay forever in $\Lambda_-$ with $x_-(s)\rightarrow\infty$ for $s\rightarrow-\infty$, we derive that there exists $s_2<s_1$ such that $\gamma_-(s_2)=(x_2,-1)$ (see the plot of $\gamma_-$ in Figure \ref{PlanosFasesCaso1orb2} left). Repeating this process indefinitely, we construct a complete, arc-length parametrized curve $\alpha_-(s)$ with infinitely many self-intersections, whose height function increases and decreases until reaching the rotation axis. Hence, the generated rotational $\hl$-surface $\Sigma_-$ is properly immersed (with self-intersections) and simply connected. See Figure \ref{PlanosFasesCaso1orb2}, bottom right.

 \begin{figure}[H]
 	\centering
 	\includegraphics[width=.7\textwidth]{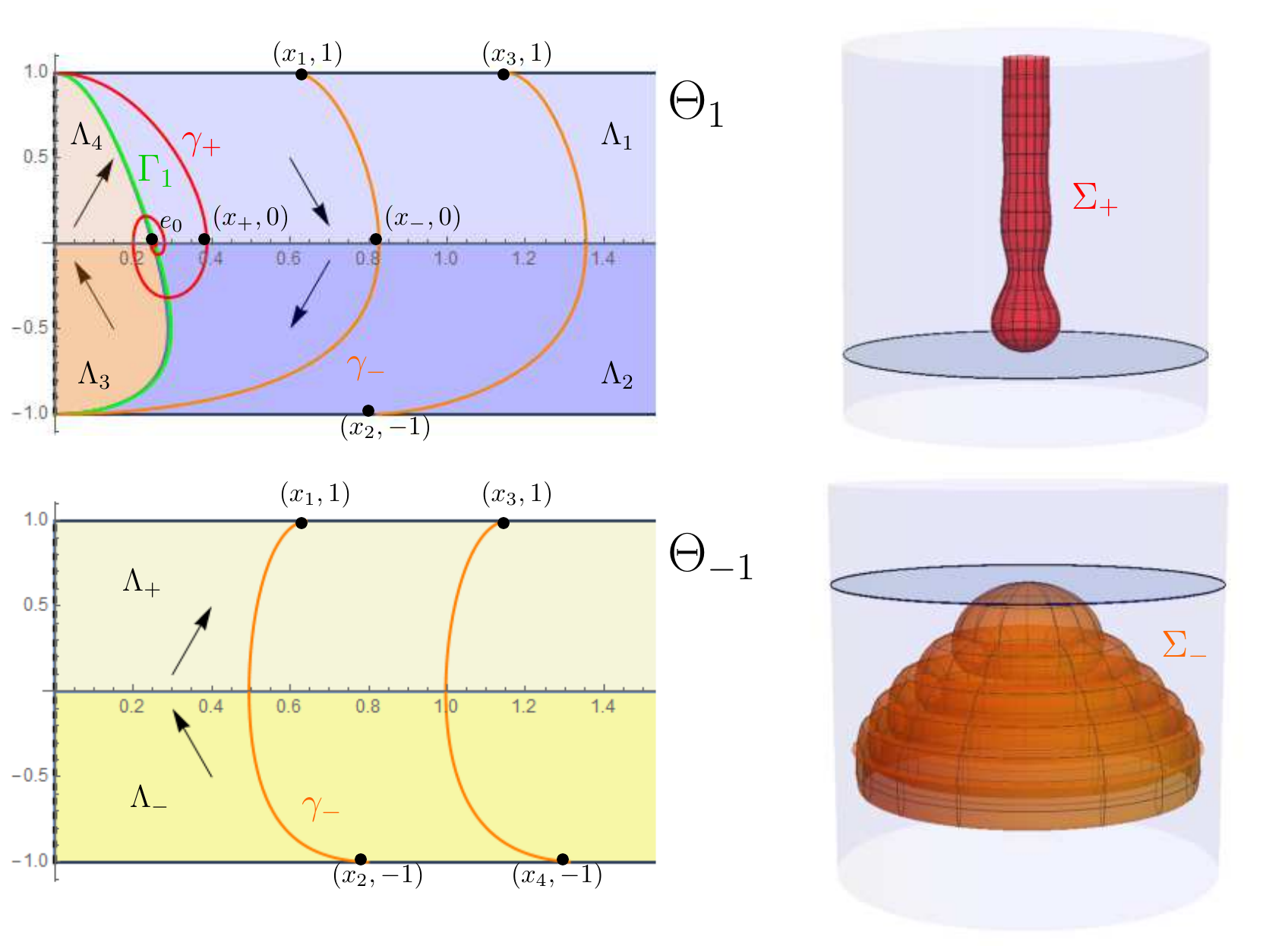}
 	\caption{Left: the phase planes $\Theta_1$ and $\Theta_{-1}$ for $\lambda>\sqrt{5}/2$ and the orbits $\gamma_+$ and $\gamma_-$, plotted in red and orange, respectively. Right: the corresponding rotational $\hl$-surfaces $\Sigma_+$ and $\Sigma_-$.}
 	\label{PlanosFasesCaso1orb2}
 \end{figure}
 
% \begin{figure}[H]
% 	\centering
% 	\includegraphics[width=.7\textwidth]{PlanoFasesCaso1orb2.pdf}
% 	\caption{Left: the phase planes $\Theta_1$ and $\Theta_{-1}$ and the orbits $\gamma_+$ and $\gamma_-$, plotted in red and orange, respectively. Right: the corresponding rotational $\hl$-surfaces $\Sigma_+$ and $\Sigma_-$.}
% 	\label{PlanosFasesCaso1orb2}
% \end{figure}

{\underline{Case $\lambda=\sqrt{5}/2$}}

Assume that $\lambda=\sqrt{5}/2$. For $\varepsilon=1$, $\Gamma_1$ is formed by two connected arcs $\Gamma_1^\pm$, each of them having the point $(0,\pm1)$ as endpoint respectively, and both having the line $y=-2/\sqrt{5}$ as an asymptote. From item $4$ in Lemma \ref{properties} we find five monotonicity regions in the phase plane $\Theta_1$ denoted by $\Lambda_1,\dots,\Lambda_5$ (see Figure \ref{PlanosFasesCaso2} left). For $\varepsilon=-1$, the phase plane $\Theta_{-1}$ is exactly the same that in the previous case $\lambda>\sqrt{5}/2$. 

From Corollary \ref{ejefase} we can assert that there exists a unique orbit $\gamma_+$ (resp. $\gamma_-$) in $\Theta_1$ with $(0,1)$ (resp. $(0,-1)$) as an endpoint.

Regarding the orbit $\gamma_+$, it converges to $e_0=(\mathrm{argtanh}\frac{1}{\sqrt{5}},0)$ as $s\rightarrow\infty$ spiraling around it infinitely many times in the same fashion as the orbit $\gamma_+$ studied in the previous case $\lambda>\sqrt{5}/2$, see Figure \ref{PlanosFasesCaso2} left. Consequently, its corresponding $\hl$-surface $\sig_+$ has the same behavior as the one shown in Figure \ref{PlanosFasesCaso1orb2}, top right.

%Firstly, we analyze the motion of $\gamma_+$ and we start by assuming that $\gamma_+(0)=(0,1)$ in $\Theta_1$. Now, the same reasoning that the one used in the case $\lambda>\sqrt{5}/2$ allows us to know that $\gamma_+(s)$ intersects the axis $y=0$ orthogonally at $(x_+,0)$ with $x_+>\arg\tanh\left(\frac{1}{\sqrt{5}}\right)$ at some finite instant $s_+>0$. Later on, $\gamma_+$ enters to $\Lambda_2$ and it remains there until intersecting the curve $\Gamma_1$, then $\gamma_+$ enters to $\Lambda_4$ and after that enters to $\Lambda_5$ and so on. As $\gamma_+$ cannot converge to any closed orbit in virtue of Theorem \ref{noorbitascerradas}, $\gamma_+$ ends up converging to $e_0$ as $s\rightarrow\infty$, spiriling around it infitely many times (see $\gamma_+$ in Figure \ref{PlanosFasesCaso2} left). Consequently, the generated surface $\Sigma_+$ is a properly embedded, simply connected $\hl$-surface which converges to $C_\lambda$ intersecting it infinitely many times. See again Figure \ref{PlanosFasesCaso1orb2} top-right.
%\blue{The structure of the orbits in $\tm1$ is the same as in the case $\lambda>\sqrt{5}/2$. They are compact arcs joining two points located at the lines $y=\pm1$.}

Now, consider the orbit $\gamma_-$ such that $\gamma_-(0)=(0,-1)$. It is clear that $\gamma_-(s)$ is totally contained in $\Lambda_3$ and when the parameter $s$ tends to $-\infty$, the orbit $\gamma_-(s)$ converges to the line $y=-2/\sqrt{5}$. Note that $\gamma_-(s)$ cannot converge to other line $y=y_0,\ y_0\in(-1,-2/\sqrt{5})$ in virtue of Proposition \ref{orbitaescapa} (see the plot of $\gamma_-$ in Figure \ref{PlanosFasesCaso2} left). Therefore, the generated $\hl$-surface $\sig_-$ is an entire, strictly convex graph whose angle function tends to the value $-2/\sqrt{5}$. See Figure \ref{PlanosFasesCaso2} right.

 \begin{figure}[H]
	\centering
	\includegraphics[width=.8\textwidth]{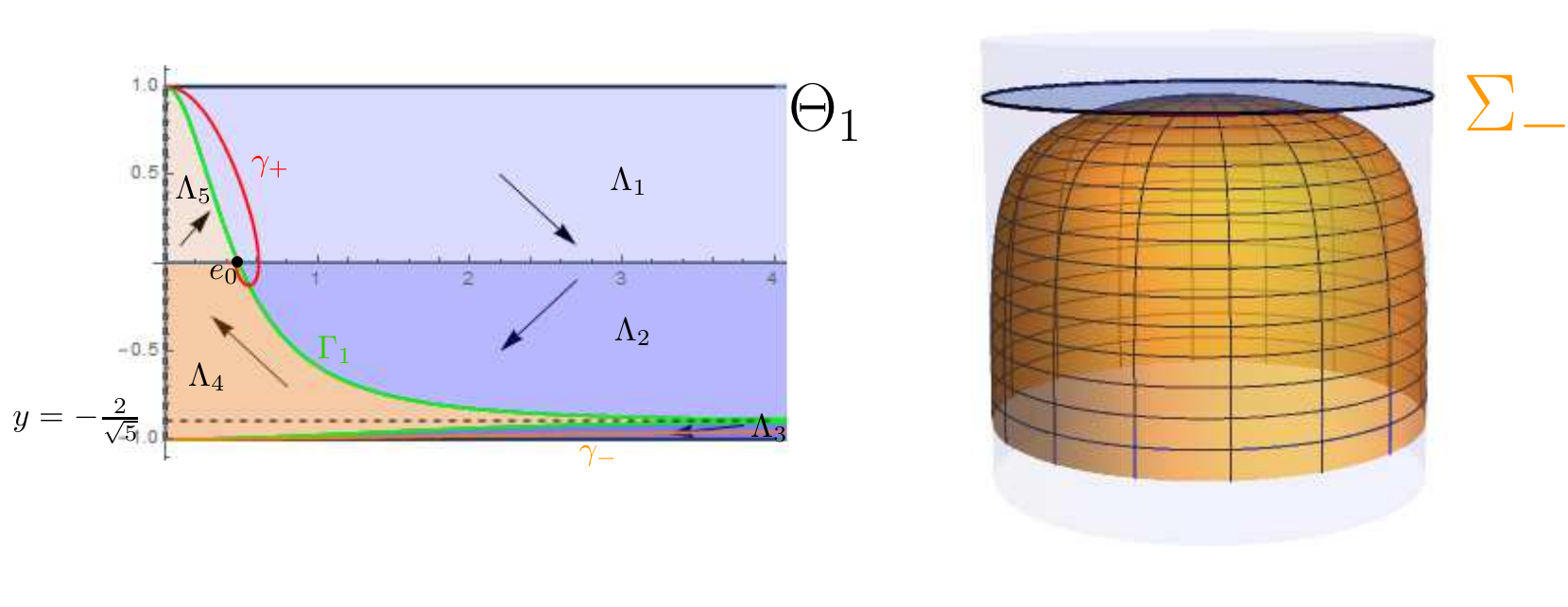}
	\caption{Left: the phase plane $\Theta_1$ for $\lambda=\sqrt{5}/2$ and the orbits $\gamma_+$ and $\gamma_-$, plotted in red and orange, respectively. Right: the corresponding rotational $\hl$-surface $\Sigma_-$.}
	\label{PlanosFasesCaso2}
\end{figure}

{\underline{Case $\lambda<\sqrt{5}/2$}}

We begin by analyzing the behavior of $\Gamma_\varepsilon$ in $\te$. Recall that the equilibrium $e_0$ exists in $\t1$ if and only if $\lambda>1/2$. Additionally, we define the candidates of asymptotes for $\Gamma_\epsilon$ as: 
 \[y_0^+:=\frac{1}{5}(-4\lambda+\sqrt{5-4\lambda^2}),\quad\textrm{and}\quad y_0^-:=\frac{1}{5}(-4\lambda-\sqrt{5-4\lambda^2}).\] 
 Let us distinguish further cases of $\lambda$:
\begin{itemize}
\item[1.] If $\lambda>1$, then $\Gamma_1$ is a disconnected arc having two connected components $\Gamma_1^+$ and $\Gamma_1^-$, with $\Gamma_+$ (resp. $\Gamma_-$) having the point $(0,1)$ (resp. $(0,-1)$) as an endpoint and the line $y=y_0^+$ (resp. $y=y_0^-$) as asymptote; see Figure \ref{FasesLambdaMenorSqrt52}. The curve $\Gamma_{-1}$ does not exist in $\tm1$ as in the previous cases.
 \begin{figure}[H]
	\centering
	\includegraphics[width=.4\textwidth]{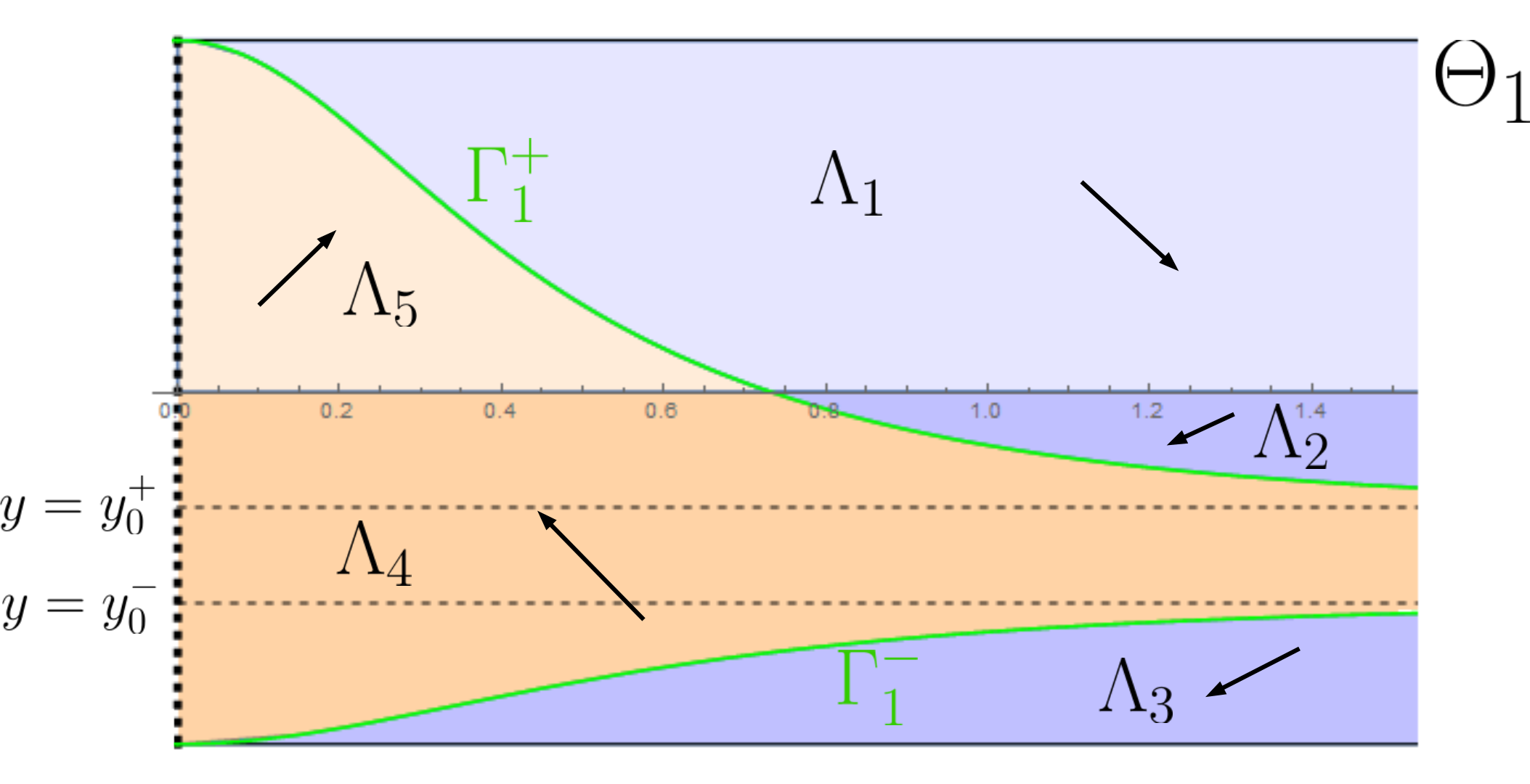}
	\caption{The phase plane $\t1$ for some $\lambda\in(1,\sqrt{5}/2)$.}
	\label{FasesLambdaMenorSqrt52}
\end{figure}
\item[2.] If $\lambda=1$, then $\Gamma_1$ is a connected arc having $(0,1)$ as endpoint and the line $y=y_0^+=-3/5$ as asymptote. The line $y=y_0^-=-1$ does not appear. The curve $\Gamma_{-1}$ does not exist in $\tm1$.
\item[3.] If $\lambda<1$, then $\Gamma_1$ is a connected arc having $(0,1)$ as endpoint and the line $y=y_0^+$ as asymptote. Moreover, $y_0^+\geq0$ if and only if $\lambda\leq1/2$. The curve $\Gamma_{-1}$ is a connected arc having $(0,-1)$ as endpoint and the line $y=y_0^-$ as an asymptote; see Figure \ref{FasesLambdaMenorSqrt521}.
 \begin{figure}[H]
	\centering
	\includegraphics[width=.8\textwidth]{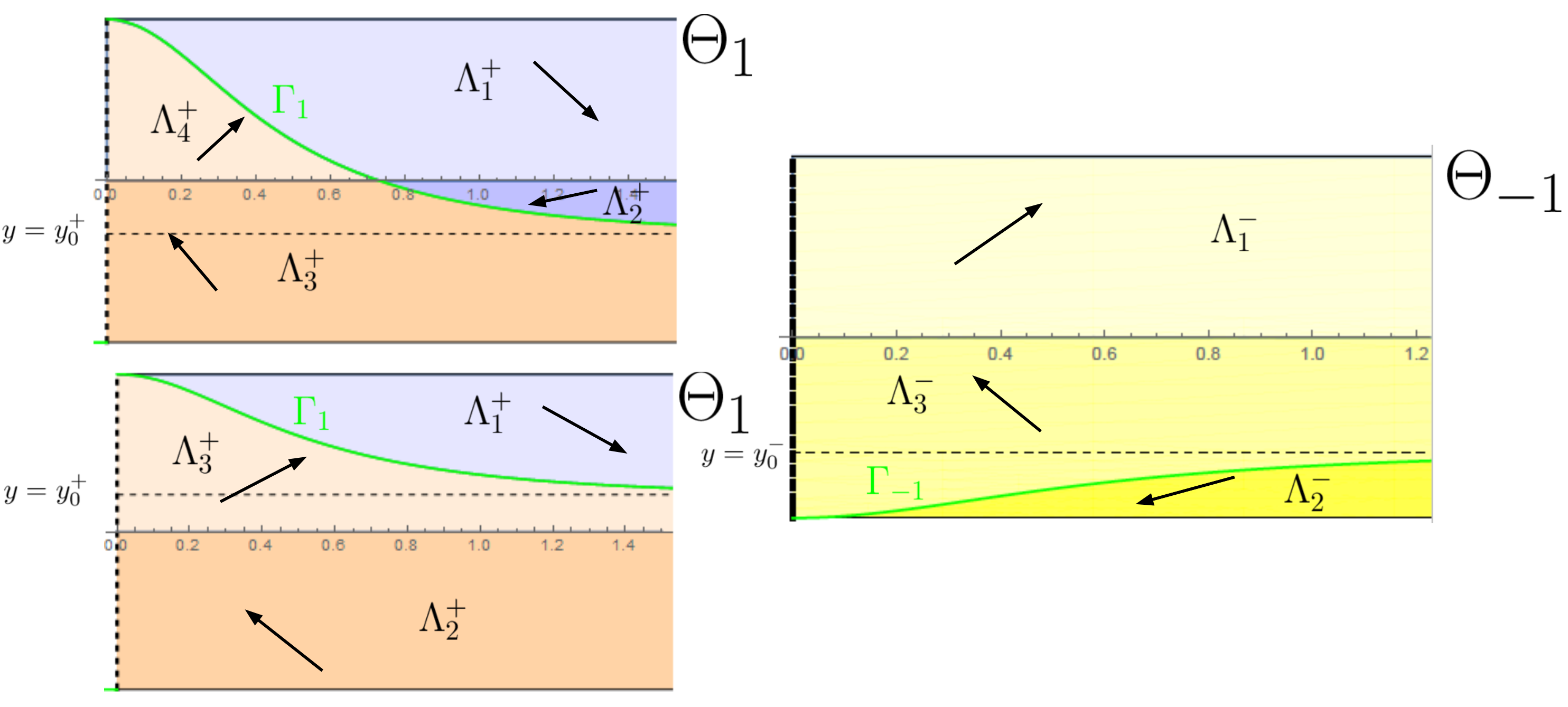}
	\caption{Top left: the phase plane $\Theta_1$ for $\lambda >1/2$. Bottom left: the phase plane $\Theta_1$ for $\lambda\leq1/2$. Right:  the phase plane $\Theta_{-1}$ for $\lambda<1$. }
	\label{FasesLambdaMenorSqrt521}
\end{figure}
\end{itemize}

%\blue{. For $\varepsilon=1$, $\Gamma_1$ is a connected arc with endpoint at $(0,1)$ and converging to the line $y=1/5(-4\lambda+\sqrt{5-4\lambda^2})$. As a matter of fact, the lowest value for this line is attained at $\lambda=\sqrt{5}/2$, whose value is precisely $y=-2\sqrt{5}$. For $\varepsilon=-1$, $\Gamma_{-1}$ is a connected arc with endpoint at $(0,-1)$ and converging to the line $y=1/5(-4\lambda-\sqrt{5-4\lambda^2})$.
%This time, the lower value for this line is attained at $\lambda=1$, whose value is precisely $y=-1$. In this critical point, the curve $\Gamma_{-1}$ does not even exist. See Figure XXX.}

Now, we study the behavior of the orbits. Once again, the existence of the orbit $\gamma_+(s)$ in $\t1$ such that $\gamma(0)=(0,1)$ follows from Corollary \ref{ejefase}. Then, if we suppose that $\lambda\leq 1/2$, $\gamma_+(s)$ stays in $\Lambda_1^+$ as it converges to the line $y=y_0^+$, and so $\sig_+$ is an entire, strictly convex graph. Otherwise, i.e., if $\lambda>1/2$, the equilibrium $e_0$ exists. Since no closed orbit exists in virtue of Theorem \ref{noorbitascerradas}, $\gamma_+$ converges to $e_0$ as $s\rightarrow\infty$, and its behavior is detailed in Lemma \ref{comportamientoequilibrio}. Consequently, $\sig_+$ is a properly embedded, simply connected $\hl$-surface and:
\begin{itemize}
	\item[$\bullet$] If $\displaystyle{\lambda>\frac{\sqrt{2}}{2}}$, $\sig_+$ intersects $C_\lambda$ infinitely many times.	
	\item[$\bullet$] If $\displaystyle{\lambda=\frac{\sqrt{2}}{2}}$, $\sig_+$ intersects $C_\lambda$ a finite number of times.
	\item[$\bullet$] If $\displaystyle{\lambda<\frac{\sqrt{2}}{2}}$, $\sig_+$ is a strictly convex graph contained in the solid cylinder bounded by $C_\lambda$ and converging asymptotically to it.
\end{itemize}

For the study of the orbit $\gamma_-(s)$ such that $\gamma_-(0)=(0,-1)$ we have to distinguish between the cases $\lambda>1,\ \lambda=1,\ \lambda<1$. This discussion will deeply influence the outcome of Corollary \ref{ejefase}:
\begin{itemize}
\item[1.] If $\lambda>1$, then $\hl(-1)>0$ and $\gamma_-(s)$ lies in $\t1$.
\item[2.] If $\lambda=1$, then $\hl(-1)=0$ and $\gamma_-(s)$ does not exist in either $\t1$ or $\tm1$.
\item[3.] If $\lambda<1$, then $\hl(-1)<0$ and $\gamma_-(s)$ lies in $\tm1$.
\end{itemize}

If $\lambda=1$, horizontal minimal planes downwards oriented are $\hh_1$-surfaces. Consequently, the uniqueness of the Cauchy problem of \eqref{1ordersys} extends to the line $y=-1$ and so no orbit can have and endpoint at this line.

%The situation in case $\lambda=1$ is that horizontal minimal planes $\mathbb{H}^2\times\{z_0\}$ with downwards orientation are $\hh_1$-surfaces. Since the prescribed function $\hh_1(y)=y+1$ is of class $C^1$ (in fact, it is analytic), the uniqueness of the Cauchy problem of \eqref{1ordersys} extends to the boundary line $y=-1$. This implies that no orbit can have an endpoint located at this line. In some sense, the line $y=-1$ itself acts as an orbit in the boundary of $\te,\ \varepsilon=\pm1$, and this orbit shall be identified with a horizontal plane with downwards orientation.

If $\lambda\neq1$, by monotonicity and Proposition \ref{orbitaescapa}, the only possibility for $\gamma_-(s)$ is to converge to the line $y=y_0^-$, and so $\sig_-$ is a downwards oriented, strictly convex, entire graph. For $\lambda>1$, the height of $\sig_-$ tends to minus infinity; for $\lambda<1$, the height of $\sig_-$ tends to infinity.

This concludes the classification of the rotational $\hl$-surfaces that intersect the axis of rotation.
\end{prooft1}

To finish, we prove Theorem \ref{teoremafueraeje}. 

\begin{prooft2} First, the equilibrium $e_0=(\mathrm{argtanh}\frac{1}{2\lambda},0)$ exists if and only if $\lambda>\frac{1}{2}$. This equilibrium generates the cylinder $C_\lambda$ with CMC equal to $\lambda$ and vertical rulings. For the remaining $\hl$-surfaces, we distinguish again three cases depending on $\lambda$. Take into account that the structure of the phase planes $\Theta_1$ and $\Theta_{-1}$ has been just studied in the previous proof. See Figure \ref{PlanosFasesCaso1} for  $\lambda>\sqrt{5}/2$, Figure \ref{PlanosFasesCaso2} for $\lambda=\sqrt{5}/2$ and Figures \ref{FasesLambdaMenorSqrt52} and \ref{FasesLambdaMenorSqrt521} for $\lambda<\sqrt{5}/2$.
	
 {\underline{Case $\lambda>\sqrt{5}/2$}}
	
Let us take $x_0>\arg\tanh(\frac{1}{2\lambda})$ and let $\gamma(s)$ be the orbit in $\Theta_1$ passing through the point $(x_0,0)$ at the instant $s=0$. For $s<0$, $\gamma$ has as endpoint some $(x_1,1),\ x_1>0$ (if $x_1=0$, $\gamma=\gamma_+$), and for $s>0$ either converges to $e_0$ as $s\rightarrow\infty$ or has another endpoint of the form $(x_2,-1),\ x_2>0$ (again, if $x_2=0$, $\gamma=\gamma_-$). In the second case, the orbit $\gamma$ continues in $\Theta_{-1}$ as a compact arc and then goes again in $\Theta_1$; see Figure \ref{LambdaMayorSqrt52}, left. After a finite number of iterations, the orbit $\gamma$ eventually converges to $e_0$ spiraling around it infinitely many times.

 \begin{figure}[H]
	\centering
	\includegraphics[width=.64\textwidth]{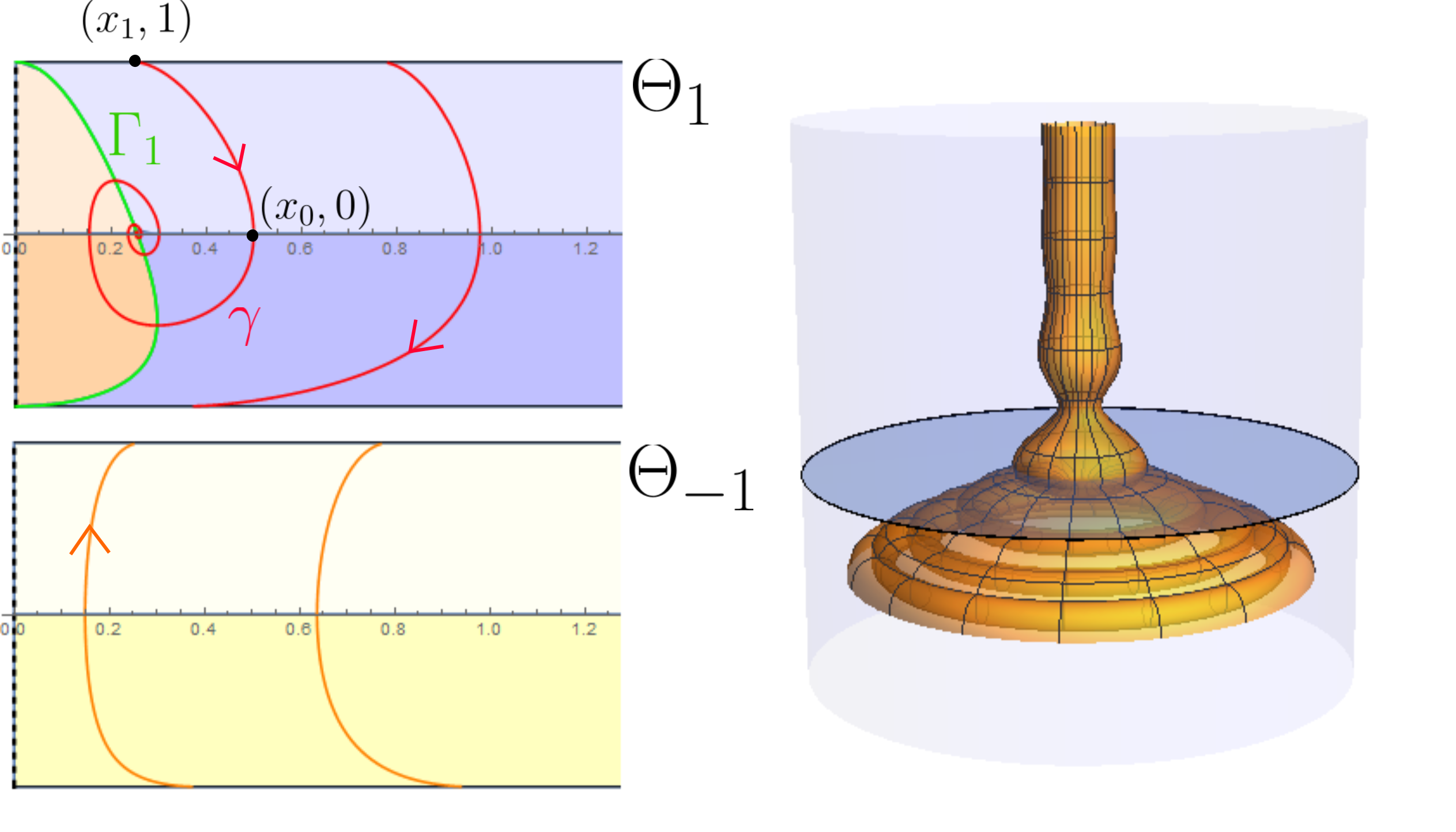}
	\caption{Left: the phase planes $\Theta_1$ and $\Theta_{-1}$ for $\lambda>\sqrt{5}/2$ and an orbit $\gamma$. Right: the rotational $\hl$-surface corresponding to $\gamma$.} 
	\label{LambdaMayorSqrt52}
\end{figure}

This configuration ensures us that the $\hl$-surface generated by $\gamma$ is properly immersed, non-embedded, and diffeomorphic to $\S^1\times\R$. One end converges to $C_\lambda$ intersecting it infinitely many times, and the other end has unbounded distance to the axis of rotation, looping and self-intersecting infinitely many times (see Figure \ref{LambdaMayorSqrt52}, right).

{\underline{Case $\lambda=\sqrt{5}/2$}}

Firstly, let us fix some $x_0>\arg\tanh(\frac{1}{\sqrt{5}})$ and consider the orbit $\gamma_1(s)$ passing through $(x_0,0)$ at $s=0$. For $s<0$, $\gamma_1(s)$ is contained in $\Lambda_1$, so it satisfies $\gamma_1(s_1)=(x_1,1)$ as endpoint for some $s_1<0$. Hence, for $s<s_1$, $\gamma_1(s)$ lies in $\tm1$ and is a compact arc whose other endpoint is located at some $\gamma_1(s_2)=(x_2,-1)$. Finally, for $s<s_2$, $\gamma_1(s)$ lies in the monotonicity region $\Lambda_3$ in $\Theta_1$ and stays there as it converges to the line $y=-2/\sqrt{5}$ as $s\rightarrow-\infty$; see Figure \ref{LambdaIgualSqrt52}, top left, the red orbit.

For $s>0$, $\gamma_1(s)$ enters to $\Lambda_2$, then goes inside $\Lambda_4$ and intersects $y=0$ for the second time in some $(\widehat{x}_0,0)$ with $\widehat{x}_0>0$. It is clear that when $x_0$ increases, then $\widehat{x}_0$ decreases and so $\widehat{x}_0\rightarrow x_\infty\geq0$ as $x_0\rightarrow\infty$. Also, note that $\gamma_1(s)$ stays always above the line $y=-2/\sqrt{5}$ since the minimum of its $y(s)$-coordinate is at the intersection of $\gamma_1(s)$ with $\Gamma_1$. In this setting, we claim that $x_\infty>0$.

Arguing by contradiction, suppose that $x_\infty=0$ and consider an orbit $\sigma(s)$ such that $\sigma(0)$ lies in $\Lambda_4$ and is located below the line $y=-2/\sqrt{5}$. Then, for $s>0$  in virtue of Proposition \ref{endpoints} and by monotonicity, $\sigma(s)$ has to reach the axis $y=0$ at a point $(\widehat{r}_1,0)$ for some finite instant. Due to the definition of $x_\infty$ and the assumption $x_\infty=0$, there exists $r>e_0$ and an orbit $\gamma$ such that $\gamma(0)=(r,0)$ and $\gamma(s_1)=(\widehat{r},0)$ with $0<\widehat{r}<\widehat{r}_1$, for some $s_1>0$. Hence, $\sigma$ and $\gamma$ intersect each other, which is a contradiction with the uniqueness of the Cauchy problem.

Secondly, let be $r_0>0$ and take $\gamma_2(s)$ the orbit such that $\gamma_2(0)=(r_0,-\sqrt{2}/5)$. For $s<0$, $\gamma_2(s)$ lies in $\Lambda_4$, intersects the curve $\Gamma_1^-$, enters $\Lambda_3$ and converges to the line $y=-2/\sqrt{5}$. For $s>0$, $\gamma_2(s)$ lies in $\Lambda_4$ until intersecting $y=0$ at some $(\tilde{r_0},0)$. Moreover, as $r_0$ increases $\tilde{r_0}$ also increases, and so $\tilde{r_0}\rightarrow r_\infty$ as $r_0\rightarrow\infty$. In particular, $r_\infty\leq x_\infty$. For $s\rightarrow\infty$, $\gamma_2(s)$ converges asymptotically to $e_0$, spiraling around it infinitely many times; see Figure \ref{LambdaIgualSqrt52}, top left, the blue orbit.

Finally, take some $\xi_0\in[r_\infty,x_\infty]$ and let $\gamma_3(s)$ be the orbit such that $\gamma_3(0)=(\xi_0,0)$. For $s>0$ it is clear that $\gamma_3(s)$ converges asymptotically to $e_0$, spiraling around it infinitely many times. Because of how $r_\infty$ and $x_\infty$ have been defined, for $s<0$ the orbit $\gamma_3(s)$ cannot intersect $\Gamma_1^-$ nor intersect $y=-2/\sqrt{5}$. Thus, the only possibility for $\gamma_3(s)$ is to converge to the line $y=-2/\sqrt{5}$ with strictly decreasing $y(s)$-coordinate; see Figure \ref{LambdaIgualSqrt52}, top left, the purple orbit.

Thus, each orbit $\gamma_i,\ i=1,2,3$ generates a properly immersed $\hl$-surface $\sig_i$, diffeomorphic to $\S^1\times\R$, with one end converging asymptotically to the CMC cylinder $C_\lambda$ and the other being a graph outside a compact set. Moreover, $\sig_1$ is non-embedded, while $\sig_2$ and $\sig_3$ have monotonous height and in particular are embedded; see Figure \ref{LambdaIgualSqrt52}, bottom.

 \begin{figure}[H]
	\centering
	\includegraphics[width=1\textwidth]{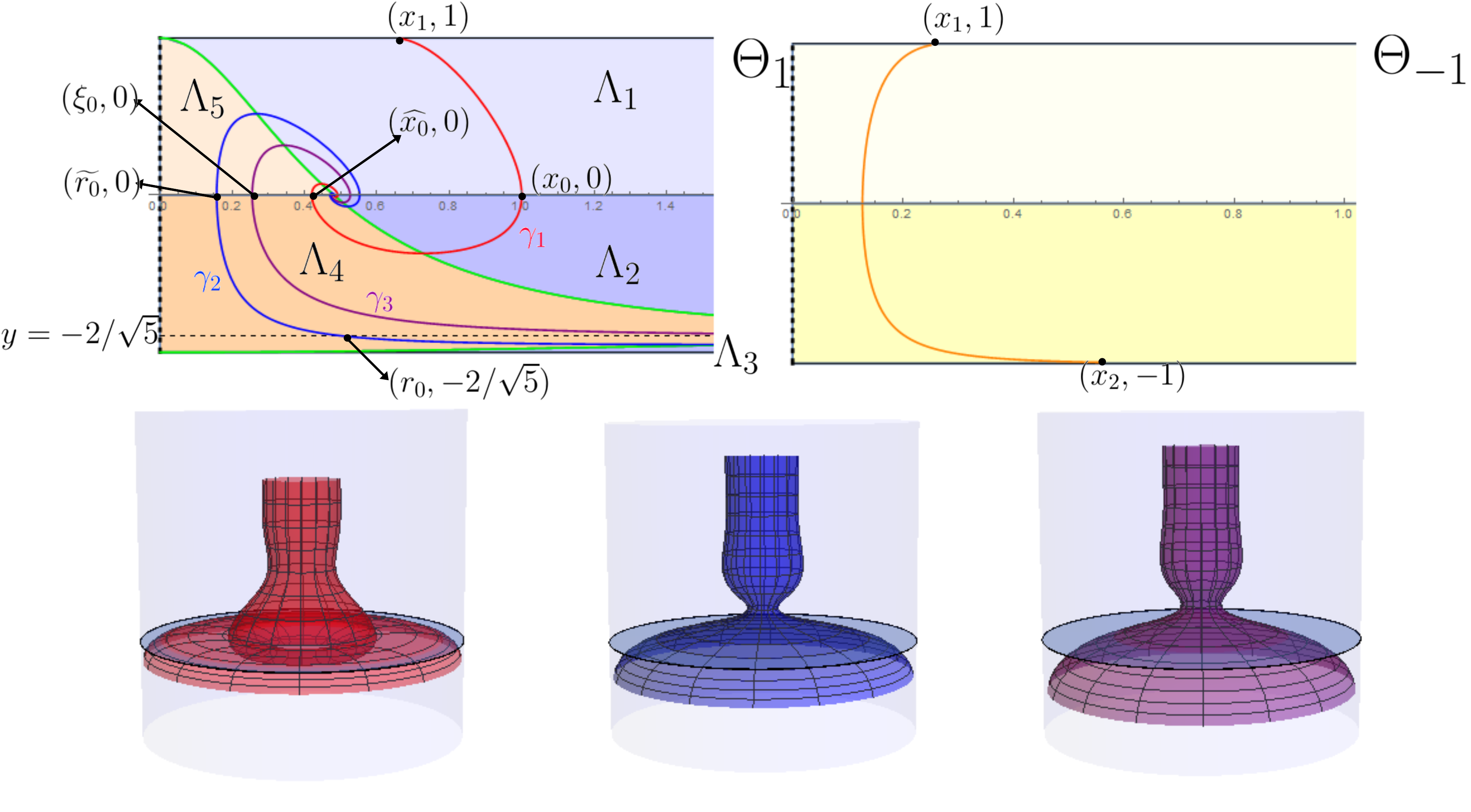}
	\caption{Top: the phase planes $\Theta_1$ and $\Theta_{-1}$ for $\lambda=\sqrt{5}/2$ and the three types of orbits $\gamma_1$, $\gamma_2$ and $\gamma_3$. Bottom: the three corresponding types of rotational $\hl$-surfaces.}
	\label{LambdaIgualSqrt52}
\end{figure}

{\underline{Case $\lambda<\sqrt{5}/2$}}

In this final case, we also distinguish between values of $\lambda$. Note that the equilibrium $e_0=(\arg\tanh(\frac{1}{2\lambda}),0)$ exists if and only if $\lambda>1/2$. Again, we define  \[y_0^+:=\frac{1}{5}(-4\lambda+\sqrt{5-4\lambda^2}),\quad\textrm{and}\quad y_0^-:=\frac{1}{5}(-4\lambda-\sqrt{5-4\lambda^2}).\] 
	
\begin{itemize}
\item[1.] \underline{Case $1\leq\lambda<\sqrt{5}/2$}. The structure of the phase plane $\t1$ (resp. $\Theta_{-1}$) is the same as the one in Figure \ref{FasesLambdaMenorSqrt52} (resp. Figure \ref{PlanosFasesCaso1}, right). Recall that $y_0^-=-1$ in $\t1$ for $\lambda=1$ and there are only four monotonicity regions. The same reasoning as in the case $\lambda=\sqrt{5}/2$ ensures us that we can construct three types of orbits (see Figure  \ref{LambdaIgualSqrt52}) and also the points $x_\infty$ and $r_\infty$:

\begin{itemize}
\item[$\bullet$] $\gamma_1(s)$ such that $\gamma_1(0)=(x_0,0)$ with $x_0>\arg\tanh(\frac{1}{2\lambda})$ and $\gamma_1(s)\rightarrow e_0$ as $s\rightarrow\infty$.
\item[$\bullet$] $\gamma_2(s)$ such that $\gamma_2(0)=(r_0,y_0^+)$ with $r_0>0$, $\gamma_2(s)\rightarrow e_0$ as $s\rightarrow\infty$ and $\gamma_2(s)\rightarrow y_0^-$ as $s\rightarrow-\infty$.
\item[$\bullet$] $\gamma_3(s)$ such that $\gamma_3(0)=(\xi_0,0)$ with $\xi_0\in[r_\infty,x_\infty]$, $\gamma_3(s)\rightarrow e_0$ as $s\rightarrow\infty$ and $\gamma_3(s)\rightarrow y_0^+$ as $s\rightarrow-\infty$.
\end{itemize}
Again, each orbit $\gamma_i,\ i=1,2,3$ generates a properly immersed $\hl$-surface $\sig_i$ that is diffeomorphic to $\S^1\times\R$, with one end converging asymptotically to the CMC cylinder $C_\lambda$ and the other being a graph outside a compact set. Moreover, $\sig_1$ self-intersects, while $\sig_2$ and $\sig_3$ have monotonous height and in particular are embedded.

%Take some $x_0>\arg\tanh(\frac{1}{2\lambda})$ and consider the orbit $\gamma_1(s)$ in $\t1$ such that $\gamma_1(0)=(x_0,0)$. Then, arguing as in the case $\lambda=\sqrt{5}/2$, for $s>0$ the orbit $\gamma_1(s)$ enters to $\Lambda_2$, intersects $\Gamma_1^+$, enters to $\Lambda_4$ and intersects $y=0$ at some $(\widehat{x}_0,0)$ before entering to $\Lambda_5$, and finally converges asymptotically to $e_0$. The same argument as in the case $\lambda=\sqrt{5}/2$ ensures us that $\widehat{x}_0\rightarrow x_\infty>0$ as $x_0\rightarrow\infty$. For $s<0$, $\gamma_1(s)$ lies in $\Lambda_1$ and has as endpoint some $\gamma_1(s_1)=(x_1,1),\ s_1<0$. For $s<s_1$, the orbit $\gamma_1(s)$ lies in $\tm1$ and hence joins $(x_1,1)$ with some $\gamma_1(s_2)=(x_2,-1)$. Finally, for $s<s_2$, $\gamma_1(s)$ lies in $\Lambda_3$ and converges to the line $y=y_0^-$ as $s\rightarrow-\infty$.
%
%At this point, we repeat the argument exhibited in the case $\lambda=\sqrt{5}/2$. We choose $r_0>0$ and $\gamma_2(s)$ the orbit in $\t1$ such that $\gamma_2(0)=(r_0,y_0^+)$. For $s<0$, $\gamma_2(s)$ converges to the line $y=y_0^-$. For $s>0$, $\gamma_2(s)$ lies in $\Lambda_4$ until intersecting $y=0$ at some $(\tilde{r_0},0)$. Again, $\tilde{r_0}\rightarrow r_\infty\leq x_\infty$ as $r_0\rightarrow\infty$.
%
%Finally, for $\xi_0\in[r_\infty,x_\infty]$ the orbit $\gamma_3(s)$ such that $\gamma_3(0)=(\xi_0,0)$ converges to $e_0$ as $s\rightarrow\infty$, lies above the line $y=y_0^+$ and converges to it as $s\rightarrow-\infty$.

\item[2.] \underline{Case $1/2<\lambda<1$}. In this case, the structure of $\Theta_1$ and $\Theta_{-1}$ is shown in Figure \ref{FasesLambdaMenorSqrt521} top left and right respectively. There are also three kind of orbits $\gamma_1(s),\gamma_2(s),\gamma_3(s)$ in $\t1$, and $\gamma_1(s)$ and $\gamma_3(s)$ behave as shown in Figure \ref{LambdaIgualSqrt52}. The only difference here is that the orbit $\gamma_2(s)$ intersects the line $y=-1$ at a finite point as $s$ decreases. Then, $\gamma_2(s)$ enters to $\tm1$ and converges to the line $y=y_0^-$ as $s\rightarrow-\infty$.

The corresponding $\hl$-surfaces are also similar to the ones constructed in the previous case.
%
%Again, each orbit $\gamma_i,\ i=1,2,3$ generates properly immersed $\hl$-surface $\sig_i$ that is diffeomorphic to $\S^1\times\R$, with one end converging asymptotically to the CMC cylinder $C_\lambda$ and the other being a graph outside a compact set. Moreover, $\sig_1$ is has a self-intersection, while $\sig_2$ and $\sig_3$ are always embedded.

%This case is quite similar to the one when $\lambda\in(1/2,1)$. The discussions are the same except for the orbits of the second kind $\gamma_2(s)$. Indeed, take some $(r_0,y_0)$ at the line $y=y_0^+$ and $\gamma_2(s)$ the orbit such that $\gamma_2(0)=(r_0,y_0)$. This time, for $s<0$ the orbit $\gamma_2(s)$ does not reach the line $y=-1$. As argued in item 2 for the case $\lambda=1$ in the proof of Theorem \ref{teoremasaleeje}, the uniqueness of the Cauchy problem of \ref{1ordersys} extends to the line $y=-1$. Therefore, $\gamma_2(s)$ converges to $y=-1$ when $s\rightarrow-\infty$. The rest of the orbits behave as either $\gamma_1(s)$ or $\gamma_3(s)$, and so their corresponding $\hl$-surfaces. Details are skipped.

\item[3.] \underline{Case $\lambda\leq 1/2$}. In this case, the structure of $\Theta_1$ and $\Theta_{-1}$ is shown in Figure \ref{FasesLambdaMenorSqrt521} bottom left and right respectively. In particular, no equilibrium point exists and the behavior of the orbits is different from the previous cases.

 \begin{figure}[H]
	\centering
	\includegraphics[width=.9\textwidth]{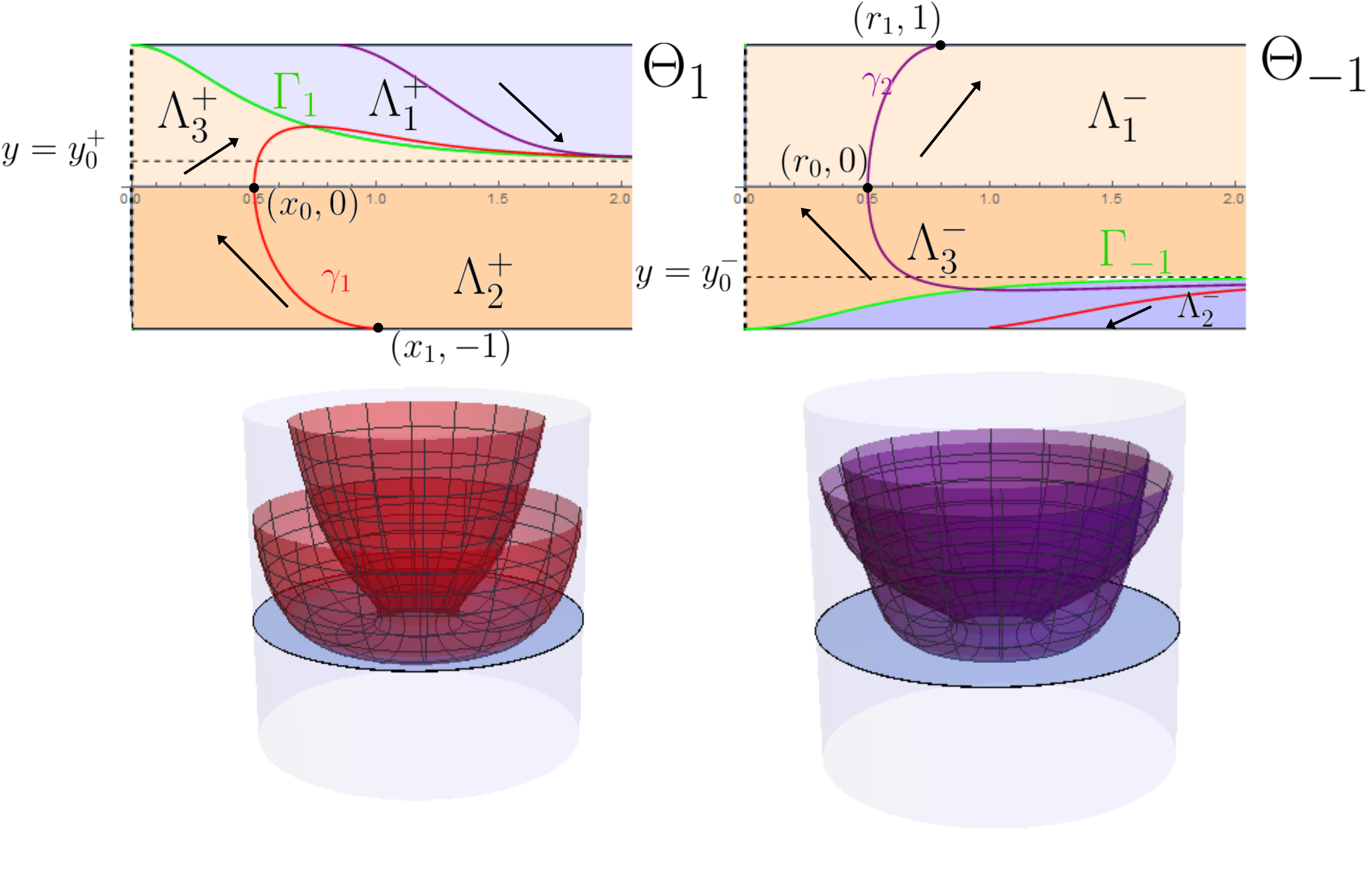}
	\caption{Top: the phase planes $\Theta_1$ and $\Theta_{-1}$ for $\lambda=1/3$. Bottom: the two corresponding $\hl$-surfaces for $\lambda=1/3$.}
	\label{LambdaMenorUnMedio}
\end{figure}

First, let be $x_0>0$ and $\gamma_1(s)$ the orbit in $\t1$ such that $\gamma_1(0)=(x_0,0)$. For $s>0$, $\gamma_1(s)$ enters to $\Lambda_3^+$, intersects $\Gamma_1$ and then lies in $\Lambda_1^+$ converging to the line $y=y_0^+$. For $s<0$, the orbit $\gamma_1(s)$ lies in $\Lambda_2^+$ and has some $\gamma_1(s_0)=(x_1,-1)$ as endpoint. Thus, $\gamma_1(s)$ for $s<s_0$ lies in $\Lambda_2^-$ and stays there converging to the line $y=y_0^-$ as $s\rightarrow-\infty$. See Figure \ref{LambdaMenorUnMedio}, the orbit in red.

Lastly, let be $r_0>0$ and $\gamma_2(s)$ the orbit in $\tm1$ such that $\gamma_2(0)=(r_0,0)$. For $s<0$, $\gamma_2(s)$ enters the region $\Lambda_3^-$ and ends up converging to the line $y=y_0^-$ as $s\rightarrow-\infty$. For $s>0$, $\gamma(s)$ enters the region $\Lambda_1^-$ and has as endpoint some $\gamma_2(s_0)=(r_1,1)$. Then, $\gamma_2(s)$ for $s>s_0$ lies in $\Lambda_1^+$ and stays there as it converges to the line $y=y_0^+$ as $s\rightarrow\infty$. See Figure \ref{LambdaMenorUnMedio}, the orbit in purple.

Again, we have two distinct $\hl$-surfaces $\sig_1$ and $\sig_2$ generated by the orbits $\gamma_1$ and $\gamma_2$. Each $\sig_i$ is properly immersed and diffeomorphic to $\S^1\times\R$, and both ends are graphs outside compact sets. Moreover, $\sig_1$ is embedded, while $\sig_2$ self intersects; see Figure \ref{LambdaMenorUnMedio}, bottom.
\end{itemize}
\end{prooft2}
% \begin{figure}[H]
%	\centering
%	\includegraphics[width=.7\textwidth]{HsupMenorUnMedio2.png}
%	\caption{The two distinct $\hl$-surfaces for $\lambda=1/3$. In red, $\sig_1$; in purple, $\sig_2$.}
%	\label{HsupLambdaMenorUnMedio}
%\end{figure}

\def\refname{References}

\end{document}